\documentclass[10pt]{amsart}
\usepackage{metalogo}

\usepackage{verbatim}

\usepackage{amsmath}
\usepackage{amsmath,amscd}
\usepackage{amsthm}
\usepackage{amsfonts}
\usepackage{amssymb}
\usepackage{mathrsfs}
\usepackage{fancyhdr}
\usepackage[all]{xy}

\usepackage{amssymb, paralist, xspace, graphicx, url, amscd, euscript, mathrsfs, stmaryrd}
\usepackage[all]{xy}

\numberwithin{equation}{section}
2
\numberwithin{subsection}{section}
\allowdisplaybreaks[1]

\newenvironment{enumerate1} {\begin{enumerate}[\upshape (1)]} {\end{enumerate}}

\theoremstyle{definition}
\newtheorem{definition}[subsubsection]{Definition}
\newtheorem{theorem}[subsubsection]{Theorem}
\newtheorem{lemma}[subsubsection]{Lemma}

\newtheorem{corollary}[subsubsection]{Corollary}

\newtheorem{Proposition}[subsubsection]{Proposition}

\newtheorem{Claim}[]{Claim}

\theoremstyle{remark}

\newtheorem{remark}{Remark}

\newcommand{\p}{\mathbf{P}^{1}}

\newcommand{\Pn}{\mathbf{P}^{n}}
\newcommand{\blA}{Bl_{Z}A}
\newcommand{\HS}{\textrm{Hilb}_{S}}
\newcommand{\HC}{\textrm{Hilb}_{C}}
\newcommand{\HuC}{\textrm{Hilb}^{n}_{\mathcal{C}|H}}
\newcommand{\Fl}{\textrm{Flag}_{S}}

\newcommand{\RH}{\mathcal{R}Hom}

\newcommand{\bl}{Bl_{Z}X}
\newcommand{\J}{\overline{J}}
\newcommand{\hig}{\mathcal{H}iggs}
\newcommand{\li}{\ell}
\newcommand{\de}{\textrm{det}}

\begin{document}
\title{Poincar\'e sheaf on the stack of rank $2$ Higgs bundles on $\p$}
\author{Mao Li}
\email{mli288@wisc.edu}

\maketitle
\begin{abstract}
Let $\hig_{ss}$ be the stack of semistable rank two Higgs bundles on $\p$ with value in $O(n)$ where $n\geq 2$. In this paper we will construct the Poincar\'e sheaf $\mathcal{P}$ on $\hig_{ss}\times_{H}\hig_{ss}$ which is flat over both components $\hig_{ss}$. This generalizes the construction of the Poincar\'e sheaf in $[1]$, $[4]$ and $[14]$.
\end{abstract}

\tableofcontents

\newpage

\section{Introduction}

\subsection{Poincar\'e sheaf}
Let $C$ be a smooth projective curve and $J$ be the Jacobian of $C$. Then it is well known that there is a Poincar\'e line bundle $\mathcal{P}$ on $J\times J$ which is the universal family of topologically trivial line bundles on $J$(See $[18]$).
When $C$ is an integral planar curve, the Jacobian $J$ is no longer projective, but we can consider the compactified Jacobian $\J$ ($[8]$,$[9]$) which parameterizes torsion free rank $1$ sheaves on $C$. In this case there is a Poincar\'e line bundle on $\mathcal{P}$ on $J\times\J$ ($[2]$) defined in the following way. Consider
\[\begin{CD}
C\times J\times\J\\
@V{\pi}VV\\
J\times\J\\
\end{CD}\]
Then:
\begin{equation}\label{Poincare line bundle}
\mathcal{P}=\de(R\pi_{*}(L\otimes F))\otimes \de(R\pi_{*}O_{C})\otimes \de(R\pi_{*}(L))^{-1}\otimes \de(R\pi_{*}(F))^{-1}
\end{equation}
where $F$ and $L$ are the universal sheaves on $C\times J$ and $C\times\J$. It is interesting to ask whether we can extend $\mathcal{P}$ to $\J\times\J$. For curves with double singularities, this has been answered in $[14]$, and the generalization to all integral planar curves is obtained in $[1]$(Similar results have also been obtained by Margarida Melo, Antonio Rapagnetta and Filippo Viviani in $[3]$ and $[4]$, where they work with moduli space instead of stacks):
\begin{theorem}
There is a maximal Cohen-Macaulay sheaf $\mathcal{P}$ on $\J\times\J$ such that the restriction of $\mathcal{P}$ to $J\times\J\bigcup\J\times J$ is the Poincar\'e line bundle given by formula~\ref{Poincare line bundle}. Moreover, $\mathcal{P}$ is flat over both component $\J$.
\end{theorem}
We shall call $\mathcal{P}$ the Poincar\'e sheaf. In fact, even though the theorem is stated only for integral curves, the argument presented in $[1]$ also works for reduced planar curves.\\

\begin{remark}
For the construction of the Poincar\'e line bundle on $J\times\J\bigcup\J\times J$, we do not need to assume that $C$ is reduced. Similarly, Lemma~\ref{equivariance property} below also holds for Poincar\'e line bundles of nonreduced planar curves.
\end{remark}

One of the main motivations for studying compactified Jacobians is that they are fibers of the Hitchin fibration. %Recall the following definition:
%\begin{definition}
%Let $X$ be a smooth projective curve and $L$ a line bundle on $X$. A rank $n$ Higgs bundle with value in $L$ is a pair $(E,\phi)$ where $E$ is a rank $n$ vector bundle and $\phi$ is a morphism of vector bundles:$E\xrightarrow{\phi} E\otimes L$.
%\end{definition}
%Denote the stack of Higgs bundles by $\hig$.
Let $X$ be a smooth projective curve, $L$ a line bundle on $X$. Denote the stack of rank $n$ Higgs bundle with value in $L$ by $\hig$. Let $H$ be the Hitchin base which parameterizes spectral curves. We have the Hitchin fibration $\hig\xrightarrow{h} H$. It is well-known that a Higgs bundle on $X$ can be naturally viewed as a torsion-free rank $1$ sheaf on the spectral curve $C$(See $[12]$). Moreover, let $H_{r}$ be the open subscheme of $H$ corresponding to reduced spectral curves. Then it is well-known that the fibers of $h$ over $H_{r}$ can be identified with the compactified Jacobian of $C$. Let $\hig^{reg}$ denote the open substack of $\hig$ corresponding to line bundles on spectral curves. Then formula~\ref{Poincare line bundle} defines a Poincar\'e line bundle $\mathcal{P}$ on $$\hig^{reg}\times_{H}\hig\bigcup\hig\times_{H}\hig^{reg}$$
Moreover, denote the open substack of Higgs bundles with generically regular Higgs field by $\widetilde{\hig}$ (Definition~\ref{generically regular Higgs field}). As we will see later in Proposition~\ref{definition of Q}, the construction in $[1]$ actually provides a maximal Cohen-Macaulay sheaf $\mathcal{P}$ on
$$\widetilde{\hig}\times_{H}\hig\bigcup\hig\times_{H}\widetilde{\hig}$$
which extends the Poincar\'e line bundle on
$$\hig^{reg}\times_{H}\hig\bigcup\hig\times_{H}\hig^{reg}$$
Moreover, it is shown in $[1]$ that $\mathcal{P}$ induces an autoequivalence of the derived category. This establishes the Langlands duality for Hitchin systems for $GL(n)$ over the locus of integral spectral curves(See $[1]$ for discussions about its relations with automorphic sheaves). Hence it is a very interesting question whether we can extend the maximal Cohen-Macaulay sheaf above to $\hig\times_{H}\hig$. The main issue here is how to extend $\mathcal{P}$ to the locus of nonreduced spectral curves. In this paper we provides a partial answer in the case of rank $2$ Higgs bundles over $\p$. Namely, let $\hig_{ss}$ be the open substack of semistable Higgs bundles. We are going to construct a maximal Cohen-Macaulay sheaf on $\hig_{ss}\times_{H}\hig_{ss}$ such that it is an extension of the Poincar\'e line bundle. It is likely that some of the constructions can be generalized to the case of higher genus curves. Also it is interesting to ask whether the Fourier-Mukai funcor induced by the Poincar\'e sheaf is still an equivalence. These will be investigated in future studies.

\subsection{Main result}
We fix an algebraically closed field $k$ of characteristic $0$. In this paper we are going to study rank $2$ Higgs bundles on $\p$ with value in $O(n)$, where $n\geq 2$. Denote the corresponding stack by $\hig$, and the open substack of semistable Higgs bundles by $\hig_{ss}$. As we mentioned in previous subsection, the construction in $[1]$ already provides a maximal Cohen-Macaulay sheaf over the locus of reduced spectral curves, so the main problem is how to extend the construction of nonreduced spectral curves. In this paper we try to provide a partial answer to this question. The main result of the paper is the following:
\begin{theorem}
There exists a maximal Cohen-Macaulay sheaf $\mathcal{P}$ on
$$\hig_{ss}\times_{H}\hig_{ss}$$
which extends the Poincar\'e line bundle on
$$(\hig^{reg}\bigcap\hig_{ss})\times_{H}\hig_{ss}\bigcup\hig_{ss}\times_{H}(\hig^{reg}\bigcap\hig_{ss})$$
Moreover, $\mathcal{P}$ is flat over the both components $\hig_{ss}$.
\end{theorem}

\begin{remark}
It is not hard to check that the complement of
$$\hig^{reg}\times_{H}\hig\bigcup\hig\times_{H}\hig^{reg}$$
has codimension greater than or equals to $3$, hence by Proposition~\ref{uniqueness of CM sheaf}, such an extension is unique if exists.
\end{remark}

In the rest of this subsection, we are going to justify the following two claims:
\begin{Claim}
Let $\hig^{(k)}$ denote the open substack of $\hig$ such that the underlying vector bundle is isomorphic to $O(k)\oplus O(k)$. Then it is enough to construct the Poincar\'e sheaf on $\hig^{(k)}\times_{H}\hig^{(l)}$ for all integers $k$ and $l$.
\end{Claim}
\begin{Claim}
In order to construct the Poincar\'e sheaf on $\hig^{(k)}\times_{H}\hig^{(l)}$, it is enough to construct the Poincar\'e sheaf on $\hig^{(0)}\times_{H}\hig^{(-n)}$.
\end{Claim}

We are going to use the following equivariance property of the Poincar\'e sheaf proved in $[1]$:
\begin{lemma}\label{equivariance property}
Let $L$ be a line bundle on $C$. Consider the automorphism $\mu_{L}$ of $\J$ defined by $F\rightarrow F\otimes L$, then we have that:
$$(\mu_{L}\times id)^{*}\mathcal{P}\simeq \mathcal{P}\otimes p_{2}^{*}(\mathcal{P}_{L})$$
where $\mathcal{P}_{L}$ is a line bundle on $\J$ obtained by restriction of the Poincar\'e line bundle to $\{L\}\times\J\hookrightarrow J\times\J$

\end{lemma}

For claim $(1)$, let $\widetilde{\hig}$ be the open substack of $\hig$ consisting of Higgs bundles with generically regular Higgs field. Since the construction in $[1]$ already provides the Poincar\'e sheaf on
$$\widetilde{\hig}\times_{H}\hig\bigcup\hig\times_{H}\widetilde{\hig}$$
(See Corollary~\ref{CM sheaf constructed in 1}), the following lemma shows that we only need to construct the Poincar\'e sheaf $\mathcal{P}$ on $\hig^{(k)}\times_{H}\hig^{(l)}$:
\begin{lemma}
$\hig_{ss}=(\widetilde{\hig}\bigcap\hig_{ss})\bigcup(\bigcup_{k}\hig^{(k)})$
\end{lemma}
The proof of this lemma will be given later (See Lemma~\ref{lemma in the introduction}).\\
For claim $(2)$, notice that tensoring by $O(m)$ gives isomorphisms between all $\hig^{(k)}$ for different $k$. So we can first construct Poincar\'e sheaf on $\hig^{(0)}\times_{H}\hig^{(-n)}$, and then use Lemma~\ref{equivariance property} to extend it to all $\hig^{(k)}\times_{H}\hig^{(l)}$. Since Poincar\'e line bundles satisfy this equivariance property, and the codimension of the complement of $\hig^{reg}\times_{H}\hig\bigcup\hig\times_{H}\hig^{reg}$ is greater than or equals to $3$, Proposition~\ref{uniqueness of CM sheaf} shows that the extension of $\mathcal{P}$ to $\hig^{(k)}\times_{H}\hig^{(l)}$ obtained in this way agrees with the Poincar\'e line bundle, and hence we are done.

For future use, let us also recall the following description about the line bundle $\mathcal{P}_{L}$:
\begin{Proposition}\label{line bundle}
Let $F$ be the universal sheaf on $\p\times\hig$, viewed also as a torsion-free rank $1$ sheaf on the spectral curve $\mathcal{C}$. Consider:
\[\begin{CD}
\p\times\hig\\
@V{p}VV\\
\hig\\
\end{CD}\]
If $L$ is a line bundle on the spectral curve $\mathcal{C}$ which is the pullback of $O(mx_{0})$ on $\p$, then we have
$$\mathcal{P}_{L}\simeq \de(p_{*}(F|_{mx_{0}}))$$

\end{Proposition}

\subsection{Review of previous work about the construction of the Poincar\'e sheaf}
In this subsection we are going to review the construction of the Poincar\'e sheaf in $[1]$, $[3]$ and $[4]$. And we will also adapt the construction to our setup. Let $C$ be a reduced planar curve embedded into a smooth surface $C\hookrightarrow S$. It is well known that $\HC^{n}$ is a complete intersection in $\HS^{n}$ of codimension $n$. Let $\overline{J}$ be the stack of torsion free rank $1$ sheaves on $C$, and $J'$ denote the open substack of torsion free rank one sheaves that are generically line bundles. The Poincar\'e sheaf on $\overline{J}\times\overline{J}$ can be constructed as follows. First, we have a natural Abel-Jacobian map:
$$\HC^{n}\xrightarrow{\alpha}  \overline{J} \qquad D\rightarrow I_{D}^{\vee}$$
Let $U_{n}$ be the open subscheme of $\HC^{n}$ given by the condition $H^{1}(I_{D}^{\vee})=0$. Then the restriction of $\alpha$ to $U_{n}$ is smooth, so we need to construct Poincar\'e sheaf on $U_{n}\times \overline{J}$ and show it descends to $\overline{J}\times \overline{J}$. Let $F$ be the universal sheaf on $C\times \overline{J}$. The Hilbert scheme of the surface is denoted by $\HS^{n}$. It is well known that $\HS^{n}$ is smooth. Let $\Fl^{n}$ be the flag Hilbert scheme of $S$, which parameterizes length $n$ subschemes together with a complete flag:
$$\emptyset=D_{0}\subseteq D_{1}\subseteq\cdots\subseteq D_{n}=D$$
Consider the following diagram:
\[\begin{CD}
\HS^{n}\times \overline{J}@<{\psi\times id}<<\widetilde{\HS^{n}}\times \overline{J}@>{\sigma_{n}\times id}>>S^{n}\times \overline{J}@<{l^{n}\times id}<<C^{n}\times \overline{J}\\
@V{p_{1}}VV\\
\HS^{n}
\end{CD}\]
where $\widetilde{\HS^{n}}$ stands for the isospectral Hilbert scheme of $S$ (See $[1]$ Proposition $3.7$ or $[17]$ for the definition). It is known that $\psi$ is finite flat of degree $n!$. Moreover, let $\HS^{' n}$ be the open subscheme of $\HS^{n}$ parameterizing subschemes that can be embedded into smooth curves, Then we have:
$$\widetilde{\HS^{n}}|_{\HS^{' n}}\simeq\Fl^{' n}=\Fl^{n}|_{\HS^{' n}}$$
Let $\mathcal{D}\xrightarrow{\pi} \HS^{n}$ be the universal finite subscheme over $\HS^{n}$ and set $\mathcal{A}=\pi_{*}O_{D}$, then we define:
\begin{equation}
Q=((\psi\times id)_{*}(\sigma_{n}\times id)^{*}(l^{n}\times id)_{*}F^{\boxtimes n})^{sign}\otimes p_{1}^{*}det(\mathcal{A})^{-1}
\end{equation}
where the upper index "sign" stands for the space of antiinvariants with respect to the natural action of the symmetric group. Then it is proved in $[1]$ that $Q$ is supported on $\HC^{n}$ and it's a maximal Cohen-Macaulay sheaf. Moreover, if we restrict $Q$ to $U_{n}$, then it descends down to $\J\times \J$. (In $[1]$ the statement is proved only for integral curves, but the same argument works for any reduced planar curves. The construction also works for families of planar curves). Let $\HS^{'n}\subseteq \HS^{n}$ be the open subscheme parameterizing subchemes $D$ such that $D$ can be embedded into a smooth curve. Then we have a simpler description of the restriction of $\widetilde{\HS^{n}}$ to $\HS^{'n}$ thanks to the following proposition ($[1]$ Proposition $3.7$):
\begin{Proposition}\label{rewrite using flag}
Let $\Fl^{n}$ be the flag Hilbert scheme of $S$, then we have
$$\widetilde{\HS^{n}}|_{\HS^{' n}}\simeq \Fl^{' n}=\Fl^{n}|_{\HS^{' n}}$$
\end{Proposition}
Hence over $\HS^{'n}$ we have $\Fl^{'n}\simeq \widetilde{\HS^{n}}$, and the construction can be written in terms of $\Fl^{'n}$.\\

We shall adapt the construction above to our setting. Namely, let $\hig$ be the stack of rank $2$ Higgs bundles on $\p$ with value in $O(n)$, and we work with the family of spectral curves over $H$:
$$\xymatrix{
\mathcal{C} \ar[rr] \ar[rd] & & S\times H \ar[ld]\\
& H\\
}
$$
Let $\widetilde{\hig}$ be the stack of Higgs bundles with generically regular Higgs field. Then we still have Abel-Jacobian map:
$$\textrm{Hilb}^{n}_{\mathcal{C}|H}\xrightarrow{\alpha}\widetilde{\hig}:\qquad D\rightarrow I_{D}^{\vee}$$
Moreover, if we set $U_{n}$ to be the open subscheme of $\textrm{Hilb}^{n}_{\mathcal{C}|H}$ consists of $D$ such that $H^{1}(\check{I}_{D})=0$, then $\alpha$ is smooth on $U_{n}$.\\
Consider the diagram:
\[\begin{CD}
\HS^{n}\times\hig@<{\psi\times id}<<\widetilde{\HS^{n}}\times\hig@>{\sigma_{n}\times id}>>S^{n}\times\hig@<{l^{n}\times id}<<\mathcal{C}^{n}\times_{H}\hig\\
@V{p_{1}}VV\\
\HS^{n}
\end{CD}\]
where $\mathcal{C}^{n}$ is $n$ fold Cartesian product of $\mathcal{C}$ over $H$:
$$\mathcal{C}^{n}=\mathcal{C}\times_{H}\cdots\times_{H}\mathcal{C}$$
Similarly, set:
\begin{equation}\label{Q}
Q=((\psi\times id)_{*}(\sigma_{n}\times id)^{*}(l^{n}\times id)_{*}F^{\boxtimes n})^{sign}\otimes p_{1}^{*}det(\mathcal{A})^{-1}
\end{equation}
Then by essentially the same argument, we get a Cohen-Macaulay sheaf $Q$ of codimension $n$ on $\HS^{n}\times\hig$. If we denote the open subscheme of $H$ corresponding to reduced spectral curves by $H_{r}$, then over $H_{r}$, the sheaf $Q$ is supported on $\textrm{Hilb}^{n}_{\mathcal{C}|H_{r}}\times_{H_{r}}\hig|_{H_{r}}$. It is not hard to check that the complement of $H_{r}$ has codimension $2n+1$. Also since $\hig$ is flat over $H$, the complement of $\HS^{n}\times\hig|_{H_{r}}$ also has codimension $2n+1$. Since $Q$ is Cohen-Macaulay of codimension $n$, the support of $Q$ is of pure dimension without embedded components and has codimension equal to $n$ in $\HS^{n}\times\hig$ (Proposition~\ref{codimension of supp of CM sheaves}). We conclude that $\textrm{Hilb}^{n}_{\mathcal{C}|H_{r}}\times_{H_{r}}\hig|_{H_{r}}$ is a dense open subset in $\textrm{Supp}(Q)$ with the codimension of the complement greater than or equals to $n+1$. Hence $Q$ is supported on $\textrm{Hilb}^{n}_{\mathcal{C}|H}\times_{H}\hig$ over the entire $H$. In conclusion, we have:
\begin{Proposition}\label{definition of Q}
Let $Q$ be the sheaf on $\HS^{n}\times\hig$ given by formula~\ref{Q}, then
\begin{enumerate1}
\item $Q$ is a maximal Cohen-Macaulay sheaf of codimension $n$ on $\HS^{n}\times\hig$ supported on $\textrm{Hilb}^{n}_{\mathcal{C}|H}\times_{H}\hig$.
\item If we consider its restriction to $U_{n}\times_{H}\hig$ (Recall that $U_{n}$ is the open subscheme of $\textrm{Hilb}^{n}_{\mathcal{C}|H}$ consists of $D$ such that $H^{1}(\check{I}_{D})=0$), then it descend to $\widetilde{\hig}\times_{H}\hig$ and agrees with the Poincar\'e line bundle on $\hig^{reg}\times_{H}\hig$.
\end{enumerate1}
\end{Proposition}

Using the fact that the complement of $\hig^{reg}$ has codimension greater than or equals to $2$, we conclude from Proposition~\ref{uniqueness of CM sheaf} that the maximal Cohen-Macaulay sheaves on $\widetilde{\hig}\times_{H}\hig$ and $\hig\times_{H}\widetilde{\hig}$ constructed from the previous Proposition glues together, hence we have the following:
\begin{corollary}\label{CM sheaf constructed in 1}
We have a maximal Cohen-Macaulay sheaf on
$$\widetilde{\hig}\times_{H}\hig\bigcup\hig\times_{H}\widetilde{\hig}$$ which agrees with the Poincar\'e line bundle on $$\hig^{reg}\times_{H}\hig\bigcup\hig\times_{H}\hig^{reg}$$
\end{corollary}

For later use, we will also denote
\begin{equation}\label{Q'}
Q'=((\sigma_{n}\times id)^{*}(l^{n}\times id)_{*}F^{\boxtimes n})\otimes (id\times\psi)^{*}(p_{1}^{*}det(\mathcal{A})^{-1})
\end{equation}
The following lemma is clear from the formula~\ref{Q}:
\begin{lemma}\label{summand}
$Q$ is a direct summand of $(\psi\times id)_{*}(Q')$
\end{lemma}

\subsection{Outline of the construction and a reformulation of the main theorem}\label{outline}
In this subsection we sketch the main idea of the construction. Let $\hig'$ be the moduli stack classifying the data $(E,\phi,s)$ where $(E,\phi)$ is a rank $2$ Higgs bundle on $\p$ with value in $O(n)$ such that $E\simeq O\oplus O$, $s$ is a nonzero global section of $E$. It is easy to see that there is a smooth surjective morphism $\hig'\rightarrow\hig^{(0)}$ (See Proposition~\ref{propositions of hig'}). We are going to construct a maximal Cohen-Macaulay sheaf on $\hig'\times_{H}\hig^{(-n)}$ and then show that it descends to $\hig^{(0)}\times_{H}\hig^{(-n)}$. First let us view $E$ as a coherent sheaf on the corresponding spectral curve $C$, and hence $s$ also gives a morphism of sheaves on $C$:
$$O_{C}\xrightarrow{s} E$$
Taking the dual of this, we get a subscheme $D'$ of $C$ over $\hig'$:
$$\mathcal{H}om_{O_{C}}(E,O_{C})\rightarrow O_{C}\rightarrow O_{D'}\rightarrow 0$$
It turns out that there is a closed substack $\mathcal{Z}$ of $\hig'$ (See Lemma~\ref{definition of Z}) such that the subscheme $D'$ is a family of length $n$ subscheme of $C$ over $\hig'\backslash\mathcal{Z}$, hence we get a morphism:
$$\hig'\backslash\mathcal{Z}\rightarrow \textrm{Hilb}^{n}_{\mathcal{C}|H}$$
It is not hard to check this morphism is smooth, hence we have a maximal Cohen-Macaulay sheaf on $\hig'\backslash\mathcal{Z}\times_{H}\hig^{(-n)}$ obtained by pulling back $Q$ (See Proposition~\ref{definition of Q}) via:
$$\hig'\backslash\mathcal{Z}\times_{H}\hig^{(-n)}\rightarrow \textrm{Hilb}^{n}_{\mathcal{C}|H}\times_{H}\hig^{(-n)}$$
In order to extend this sheaf to $\hig'\times_{H}\hig^{(-n)}$, we need to find a way to resolve the rational map $\hig'\dashrightarrow \textrm{Hilb}^{n}_{\mathcal{C}|H}$. It turns out that if we denote the blowup of $\hig'$ along $\mathcal{Z}$ by $\hig''$, then this blowup resolves the rational map(See Proposition~\ref{resolve the rational map}) and hence we have the following diagram:
$$\xymatrix{
\hig''\times_{H}\hig^{(-n)} \ar[r]^{g} \ar[d]^{\pi} & \textrm{Hilb}^{n}_{\mathcal{C}|H}\times_{H}\hig^{(-n)}\\
\hig'\times_{H}\hig^{(-n)}\\
}
$$
The main theorem of the paper can be restated as follows:
\begin{theorem}\label{main theorem}
Consider the diagram:
$$\xymatrix{
\hig''\times_{H}\hig^{(-n)} \ar[r]^{g} \ar[d]^{\pi} & \textrm{Hilb}^{n}_{\mathcal{C}|H}\times_{H}\hig^{(-n)}\\
\hig'\times_{H}\hig^{(-n)}\\
}
$$
Then we have:
\begin{enumerate1}
\item $g^{*}(Q)$ is a maximal Cohen-Macaulay sheaf on $\hig''\times_{H}\hig^{(-n)}$. Denote it by $M$.
\item $R\pi_{*}(M)$ is a maximal Cohen-Macaulay sheaf on $\hig'\times_{H}\hig^{(-n)}$, i.e. $R^{i}\pi_{*}(M)=0$ for $i>0$ and $\pi_{*}(M)$ is a maximal Cohen-Macaulay sheaf.
\item $\pi_{*}(M)$ descends to $\hig^{(0)}\times_{H}\hig^{(-n)}$ and agrees with the Cohen-Macaulay sheaf in Corollary~\ref{CM sheaf constructed in 1} when restricted to
$$(\widetilde{\hig}\bigcap\hig^{(0)})\times_{H}\hig^{(-n)}$$
Hence we have a maximal Cohen-Macaulay sheaf on $\hig_{ss}\times_{H}\hig_{ss}$
\item The maximal Cohen-Macaulay sheaf in $(3)$ is flat over both components $\hig_{ss}$.
\end{enumerate1}

\end{theorem}

Part $(1)$ will be established in Corollary~\ref{part 1 of the main theorem}, Part $(2)$ will be proved in the last section. Assume them for the moment, let us prove $(3)$ and $(4)$:
\begin{proof}
Part $(4)$ is a direct consequence of Proposition~\ref{smoothness of semistable higgs bundles}, Proposition~\ref{properties of Hitchin fibration} and Proposition~\ref{flatness of CM sheaves}. The proof of $(3)$ is again based on Proposition~\ref{uniqueness of CM sheaf}. From the construction itself, it is clear that on $(\hig'\backslash\mathcal{Z})\times_{H}\hig^{(-n)}$, $M$ is the pullback of $Q$ on $\textrm{Hilb}^{n}_{\mathcal{C}| H}\times_{H}\hig^{(-n)}$. Hence by Proposition~\ref{definition of Q}, on $(\hig'\backslash\mathcal{Z})\times_{H}\hig^{(-n)}$, $M$ descends and agrees with the maximal Cohen-Macaulay sheaf in Corollary~\ref{CM sheaf constructed in 1}. Because $\mathcal{Z}$ is a regular embedding of codimension greater than or equals to $3$ in $\hig'$ (Lemma~\ref{definition of Z}), and the Hitchin fibration $\hig^{(-n)}\rightarrow H$ is flat, we see that $\mathcal{Z}\times_{H}\hig^{(-n)}$ is also a regular embedding of codimension greater than or equals to $3$ in $\mathcal{Z}\times_{H}\hig^{(-n)}$, hence by Proposition~\ref{uniqueness of CM sheaf}, the descent data extends to $\hig'\times_{H}\hig^{(-n)}$. This finishes the proof.

\end{proof}

Let us now give an overview of the organization of the paper. In Section $2$ we are going to review some preliminary results that will be needed later. In Section $3$ we will first study the geometric properties of the stack of Higgs bundles and the Hitchin fibration. In particular, we are going to construct the following key diagram and study its geometric properties in subsection $3.3$ and $3.4$:
$$\xymatrix{
\hig''\times_{H}\hig^{(-n)} \ar[r]^{g} \ar[d]^{\pi} & \textrm{Hilb}^{n}_{\mathcal{C}|H}\times_{H}\hig^{(-n)}\\
\hig'\times_{H}\hig^{(-n)}\\
}
$$
The main results are Proposition~\ref{resolve the rational map} and Theorem~\ref{properties of g}. In Section $4$ we will prove a cohomological vanishing result that will be used to prove part $(2)$ of Theorem~\ref{main theorem}. The main result is Lemma~\ref{cohomological vanishing}. In the last section we are going to finish the proof of Theorem~\ref{main theorem}.

\subsection{Acknowledgments}
I'm very greatful to my advisor Dima Arinkin for introducing me to this facinating subject as well as his encouragement and support throughout the entire project. This work is supported by NSF grant DMS-1603277.

\section{Preliminaries}

\subsection{Local complete intersection morphism and relative complete intersection morphism}
In this section we collect some facts about local complete intersection morphism and relative complete intersection morphisms that will be used in this paper.
\begin{definition}
\begin{enumerate1}
\item A morphism of schemes $X\xrightarrow{f} Y$ locally of finite presentation is called local complete intersection morphism if for any $x\in X$, there exist open neighborhoods $U$ of $x$ and $V$ of $f(x)$ such that $U\xrightarrow{f} V$ can be factored as:
    $$\xymatrix{
    U \ar[rr]^{i} \ar[rd]^{f} & & W \ar[ld]^{g}\\
     & V\\
     }
     $$
    such that $i$ is a regular embedding and $g$ is smooth.
\item $f$ is called relative complete intersection morphism if it is a local complete intersection morphism and it is flat.
\end{enumerate1}

\end{definition}

The following properties are well known (See $[7]$, Tag $068$E and Tag $01$UB):
\begin{Proposition}\label{properties of lci}
\begin{enumerate1}
\item Local complete intersection morphisms are preserved under composition and flat base change. Relative complete intersection morphisms are preserved under composition and base change.
\item If $f$ is a local complete intersection morphism, then $f$ is of finite tor-dimension.
\item If $X\xrightarrow{f} Y$ is flat, then $f$ is a relative complete intersection iff each fiber $X_{y}$ is a local
    complete intersection over $k(y)$ for any $y\in Y$.
\item Consider:
$$\xymatrix{
X \ar[rr]^{f} \ar[rd]^{h} & & Y \ar[ld]^{g}\\
& S\\
}
$$
Assume $h$ is a local complete intersection morphism, $g$ is smooth, then $f$ is also a local complete intersection morphism.
\item Let \[\begin{CD}
W@>>>X\\
@VVV@V{f}VV\\
Z@>>>Y\\
\end{CD}\]
be a Cartesian diagram. Assume that $f$ is a local complete intersection morphism, $Z\hookrightarrow Y$ is a regular embedding of codimension $n$ such that $W$ is also a regular embedding of codimension $n$ in $X$. Then $W\rightarrow Z$ is also a local complete intersection morphism.
\end{enumerate1}

\end{Proposition}

\begin{Proposition}\label{criterion for relative complete intersection}
Let $X$ and $Y$ be finite type pure dimensional schemes over a field $k$. Assume $X$ is local complete intersection over $k$ and $Y$ is smooth over $k$. Let $X\xrightarrow{f} Y$ be a morphism of schemes, then $f$ is a relative complete intersection morphism iff each each fiber $X_{y}$ has dimension less than or equals to $\textrm{dim}(X)-\textrm{dim}(Y)$.

\end{Proposition}

\begin{proof}
The only nontrivial part is the "if" part. Suppose the dimension of each fiber is less than or equals to $\textrm{dim}(X)-\textrm{dim}(Y)$, then it is well-known that in this case $f$ is flat ($[6]$ Theorem $18.16$ part $(2)$). Since $Y$ is smooth, the flatness of $f$ implies that each fiber is also a local complete intersection, hence $f$ is a relative complete intersection morphism by Part$(3)$ of Proposition~\ref{properties of lci}.

\end{proof}

\subsection{Cohen-Macaulay sheaves}
In this section we shall review some facts about Cohen-Macaulay sheaves that will be used freely in the paper. For simplicity we shall work with a Gorenstein schemes $X$ (since all schemes(stacks) appearing in the paper are Gorenstein), so that the dualizing complex of $X$ can be taken to be $O_{X}$.
Most of these properties can be found in $[15]$.

\begin{definition}
Let $M$ be a coherent sheaf on $X$ such that $\textrm{codim} (\textrm{Supp(M)})=d$. Then $M$ is called Cohen-Macaulay of codimension $d$ if $\RH_{O_{X}}(M,O_{X})$ sits in degree $d$. In particular, $M$ is a Cohen-Macaulay module. $M$ is called maximal Cohen-Macaulay if $d=0$
\end{definition}
Notice that if $d=0$ then both $M$ and $\mathcal{H}om_{O_{X}}(M,O_{X})$ are maximal Cohen-Macaulay, and the functor $\mathcal{H}om_{O_{X}}(-,O_{X})$ induces an anti-auto-equivalence in the category of maximal Cohen-Macaulay sheaves.

\begin{Proposition}\label{codimension of supp of CM sheaves}
If $M$ is a Cohen-Macaulay sheaf of codimension $d$ on a Gorenstein scheme $X$, then the support of $M$ is of pure dimension without embedded components, and $\textrm{codim}(\textrm{Supp}(M))=d$. Moreover, the functor $M\rightarrow \mathcal{E}xt^{d}_{O_{X}}(M,O_{X})$ induces an anti-auto-equivalence of the category of Cohen-Macaulay sheaves of codimension $d$.
\end{Proposition}

We will use  the following property of maximal Cohen-Macaulay sheaves ($[1]$ Lemma $2.2$):
\begin{Proposition}\label{uniqueness of CM sheaf}
Let $X$ be a Gorenstein scheme of pure dimension, $M$ a maximal Cohen-Macaulay sheaf and $Z\subseteq X$ a closed subscheme of codimension greater than or equals to $2$, then we have $M\simeq j_{*}(M|_{X-Z})$ for $j$: $X-Z\hookrightarrow X$. Hence for any two maximal Cohen-Macaulay sheaves $M$ and $N$, we have $Hom_{X}(M,N)\simeq Hom_{X\backslash Z}(M|_{X\backslash Z},N|_{X\backslash Z})$

\end{Proposition}

The following is immediate from definitions:
\begin{corollary}\label{resolution of CM sheaf}
Let $M$ be a maximal Cohen-Macaulay sheaf on $X$, $\mathcal{E}$ a vector bundle of rank $r$. Let $s$ be a global section of $\mathcal{E}$ such that $s$ defines a regular sequence of length $r$ on $X$. Denote the vanishing locus of $s$ by $Z$, then we have resolutions of the form:
$$\bigwedge(\check{\mathcal{E}})\otimes M\rightarrow M|_{Z}$$

\end{corollary}

\begin{Proposition}\label{pullback of CM sheaf}
Let $X\xrightarrow{f} Y$ be a morphism between Gorenstein schemes of finite tor-dimension. Let $Z\hookrightarrow Y$ be a regular embedding of codimension $d$, and denote the pullback of $Z$ by $f$ by $W$. Suppose $W$ is also a regular embedding of codimension $d$ on $X$, then:
\begin{enumerate1}
\item $W\rightarrow Z$ is also of finite tor-dimension.
\item For any maximal Cohen-Macaulay sheaf $M$ on $Z$, we have $Lf^{*}M=f^{*}M$, and $f^{*}M$ is also a maximal Cohen-Macaulay sheaf on $W$.
\item $Lf^{*}\RH_{O_{Z}}(M,O_{Z})\simeq \RH_{O_{W}}(f^{*}M,O_{W})$
\end{enumerate1}

\end{Proposition}

\begin{proof}
For part $(1)$, since the statement is local on $X$ and $Y$, we can reduce it to the following case: \\
Suppose $B$ is an algebra over $A$ and $B$ has finite tor-dimension over $A$. Let $(x_{1},\cdots,x_{r})$ be a length $r$ regular sequence on $A$, such that they also form a regular sequence of $B$, then $B/(\mathbf{x})$ also has finite tor-dimension over $A/(\mathbf{x})$\\
By assumption, we have a bounded flat resolution $F^{*}\rightarrow B$ of $B$ as an $A$ module. Since $(\mathbf{x})$ is a regular sequence on $B$, we have that $B\otimes^{L}_{A} A/(\mathbf{x})\simeq B/(\mathbf{x})$, hence $F^{*}\otimes_{A}A/(\mathbf{x})$ is a bounded flat resolution of $B/(\mathbf{x})$. This proves $(1)$\\

For $(2)$, let $F^{*}$ be a bounded flat resolution of $B$ as an $A$ module. Then $Lf^{*}M$ is represented by the complex
$$F^{*}\otimes_{A}M\simeq F^{*}\otimes_{A}A/(\mathbf{x})\otimes_{A/(\mathbf{x})}M\simeq B/(\mathbf{x})\otimes^{L}_{A/(\mathbf{x})}M$$
By part $(1)$, $B/(\mathbf{x})$ is of finite tor-dimension over $A/(\mathbf{x})$, this reduces us to the case when $Z=Y$ and $M$ is a maximal Cohen-Macaulay sheaf on $Y$. And this is established in $[1]$ Lemma $2.3$.\\

For $(3)$, we can again reduce it to the case when $Z=Y$ using $(1)$ and $(2)$. So Assume now that $M$ is a maximal Cohen-Macaulay sheaf on $Y$. We have
$$\RH_{O_{Y}}(M,O_{Y})=\mathcal{H}om_{O_{Y}}(M,O_{Y})$$
are maximal Cohen-Macaulay sheaves, so we only need to show:
$$f^{*}\mathcal{H}om_{O_{Y}}(M,O_{Y})\simeq \mathcal{H}om_{O_{X}}(f^{*}M,O_{X})$$
The statement is local, so we only need to show the following: Let $A\xrightarrow{f} B$ be a morphism of finite tor-dimension between Gorenstein rings, and $M$ a maximal Cohen-Macaulay module of $A$, then we have
$$Hom_{A}(M,A)\otimes_{A} B\simeq Hom_{B}(M\otimes_{A} B,B)$$
Choose a presentation of $M$:
$$Q\rightarrow P\rightarrow M\rightarrow 0$$
where $P$ and $Q$ are free of finite rank. We can split it into two exact sequences:
$$0\rightarrow K\rightarrow P\rightarrow M\rightarrow 0$$
$$0\rightarrow N\rightarrow Q\rightarrow K\rightarrow 0$$
Also all modules involved are maximal Cohen-Macaulay modules. So we have exact sequences:
$$0\rightarrow Hom_{A}(M,A)\rightarrow Hom_{A}(P,A)\rightarrow Hom_{A}(K,A)\rightarrow 0$$
$$0\rightarrow Hom_{A}(K,A)\rightarrow Hom_{A}(Q,A)\rightarrow Hom_{A}(N,A)\rightarrow 0$$
By part $(1)$, we have exact sequences:
$$0\rightarrow Hom_{A}(M,A)\otimes_{A}B\rightarrow Hom_{A}(P,A)\otimes_{A}B\rightarrow Hom_{A}(K,A)\otimes_{A}B\rightarrow 0$$
$$0\rightarrow Hom_{A}(K,A)\otimes_{A}B\rightarrow Hom_{A}(Q,A)\otimes_{A}B\rightarrow Hom_{A}(N,A)\otimes_{A}B\rightarrow 0$$
Hence the following is exact:
$$0\rightarrow Hom_{A}(M,A)\otimes_{A}B\rightarrow Hom_{A}(P,A)\otimes_{A}B\rightarrow Hom_{A}(Q,A)\otimes_{A}B$$
Now since
$$Q\rightarrow P\rightarrow M\rightarrow 0$$
is exact, we have:
$$Q\otimes_{A}B\rightarrow P\otimes_{A}B\rightarrow M\otimes_{A}B\rightarrow 0$$
Hence:
$$0\rightarrow Hom_{B}(M\otimes_{A}B,B)\rightarrow Hom_{B}(P\otimes_{A}B,B)\rightarrow Hom_{B}(Q\otimes_{A}B,B)$$
So we have the following morphisms between exact sequences:
$$\xymatrix{
0 \ar[r] & Hom_{A}(M,A)\otimes_{A}B \ar[r] \ar[d]^{\alpha} & Hom_{A}(P,A)\otimes_{A}B \ar[r] \ar[d]^{\beta} & Hom_{A}(Q,A)\otimes_{A}B \ar[d]^{\gamma}\\
0 \ar[r] & Hom_{B}(M\otimes_{A}B,B) \ar[r] & Hom_{B}(P\otimes_{A}B,B) \ar[r] & Hom_{B}(Q\otimes_{A}B,B)\\
}
$$
Since $P$ and $Q$ are free of finite rank, $\gamma$ and $\beta$ are isomorphisms, so $\alpha$ is also an isomorphism.

\end{proof}

\begin{Proposition}\label{flatness of CM sheaves}
Let $X$ be a Gorenstein scheme and $Y$ be a smooth scheme. Let $X\xrightarrow{f} Y$ be a flat morphism. If $M$ is a maximal Cohen-Macaulay sheaf on $X$, then $M$ is flat over $Y$

\end{Proposition}

\begin{proof}
We can reduce the claim to the following statement: Let
$$(A,\mathfrak{m})\xrightarrow{f} (B,\mathfrak{n})$$
is a local flat morphism of noetherian local rings such that $A$ is regular and $B$ is Gorenstein. If $M$ is a maximal Cohen-Macaulay $B$ module, then $M$ is flat over $A$. We can apply the local criterion for flatness: Let $\mathfrak{m}=(x_{1},\cdots,x_{r})$, since $f$ is flat, $(x_{1},\cdots,x_{r})$ is also a regular sequence on $B$, hence it is a regular sequence of $M$. Since $Tor_{A}(A/\mathfrak{m},M)$ is computed by the Koszul complex
$$K(x_{1},\cdots,x_{r})\otimes M$$
It follows that $Tor^{i}_{A}(A/\mathfrak{m},M)=0$ for $i>0$, hence $M$ is flat over $A$ by the local criterion for flatness.

\end{proof}

\subsection{Blowup and cohomology}
In this section we gather some facts about blowups which will be used in the proof of the main theorem. Let $X$ be a scheme, $\mathcal{E}$ a vector bundle of rank $n+1$ on $X$ with a global section $s\in H^{0}(X,\mathcal{E})$, such that the vanishing locus of $s$ is a regular embedding of codimension $n+1$. Denote the vanishing locus of $s$ by $Z$. Then we have the following description about the blowup of $X$ along $Z$ (See Charpter $11$ of $[11]$):
\begin{Proposition}\label{blowup as regular embedding}

The blowup of $X$ along $Z$ is a regular embedding of codimension $n$ in $\textbf{P}(\mathcal{E})$:

$$
\xymatrix{
  Bl_{Z}X \ar[rr]^{i} \ar[dr]^{p}
                &  &    \textbf{P}(\mathcal{E})=\textrm{Proj}(\textrm{Sym}_{X}\check{\mathcal{E}}) \ar[dl]^{\pi}    \\
                & X
}
$$

It can be described by the following: On $\mathbf{P}(\mathcal{E})$ we have a natural morphism of vector bundles:
$$\pi^{*}\mathcal{E}\xrightarrow{\varphi} T_{P(\mathcal{E})\mid X}\otimes O(-1)$$
The blowup is the vanishing locus of $\varphi(\pi^{*}s)\in H^{0}(T_{P(\mathcal{E})\mid X})$. It is a regular embedding in $P(\mathcal{E})$ so we have a Koszul resolution:
\begin{equation}\label{bl}
\bigwedge^{*}(\Omega_{P(\mathcal{E})\mid X}\otimes O(1))\rightarrow O_{Bl_{Z}X}
\end{equation}
\end{Proposition}

From the description of the blowup as a closed subscheme in $\mathbf{P}(\mathcal{E})$, we get the following moduli interpretation of the blowup:
\begin{Proposition}\label{moduli interpretation of bl}
Let $S\xrightarrow{q} X$ be a scheme over $X$, then there is a bijection between $Hom_{X}(S,Bl_{Z}X)$ and the line subbundles $L$ of $q^{*}(\mathcal{E})$(Meaning that the quotient $q^{*}(\mathcal{E})/L$ is a vector bundle) such that $O\xrightarrow{q^{*}(s)} q^{*}(\mathcal{E})$ factors through $L$. The universal object on $Bl_{Z}X$ is given by
$$O(E)\hookrightarrow p^{*}(\mathcal{E})$$
where $E$ is the exceptional divisor on $Bl_{Z}X$ and $p$ is the projection $Bl_{Z}X\xrightarrow{p} X$.
\end{Proposition}

\begin{proof}
Notice that the data $L\hookrightarrow q^{*}(\mathcal{E})$ already defines a morphism:
$$\xymatrix{
S \ar[rr]^{f} \ar[dr]^{q} &  & \mathbf{P}(\mathcal{E})\ar[ld]^{\pi}\\
 & X\\
}
$$
By our discussions above, the blowup is given by the vanishing locus of the section:
$$O\xrightarrow{\pi^{*}(s)} \pi^{*}\mathcal{E}\rightarrow T_{P(\mathcal{E})|X}\otimes O(-1)$$
We have the following natural exact sequence on $\mathbf{P}(\mathcal{E})$:
$$0\rightarrow O(-1)\rightarrow\pi^{*}\mathcal{E}\rightarrow T_{P(\mathcal{E})|X}\otimes O(-1)\rightarrow 0$$
Also the pullback of $O(-1)\rightarrow\pi^{*}\mathcal{E}$ to $S$ is $L\rightarrow q^{*}\mathcal{E}$. So to prove $f$ factors through the blowup, we only need to show $q^{*}(s)$ factors through $L$, which is the second assumption in the claim.

\end{proof}

We shall apply the above proposition to the following situation:
\begin{corollary}\label{bl resolve nonflatness}
Let $Y\xrightarrow{f} X$ be a proper flat morphism of schemes with geometrically integral fibers. Let $L$ be a line bundle on $Y$ such that $f_{*}(L)$ is a vector bundle of rank $n+1$ and the formation of $f_{*}(L)$ commutes with arbitrary base change. Let $s$ is a global section of $L$ such that the vanishing locus of the induced section $t$ of $f_{*}(L)$ has codimension $n+1$ on $X$. Let $Z$ be the vanishing locus of $t$, Set $Y'=Y\times_{X}Bl_{Z}X$. Consider:
\[\begin{CD}
Y'@>{g'}>>Y\\
@V{f'}VV@V{f}VV\\
Bl_{Z}X@>{g}>>X\\
\end{CD}\]
Then We have:
\begin{enumerate1}
\item The vanishing locus of $s$ on $Y$ is a relative effective Cartier divisor over the open subset $X\backslash Z$.
\item The section $s$ extends to a morphism $O(E)\xrightarrow{s'} g^{'*}(L)$ on $Y'$ such that the vanishing locus of $s'$ is a relative effective Cartier divisor over $Bl_{Z}X$ which extends the relative effective Cartier divisor over $X\backslash Z$
\end{enumerate1}
\end{corollary}

\begin{proof}

By our assumption, since the morphism $O_{X}\xrightarrow{t} f_{*}(L)$ induced by the global section $s$ is nonvanishing over $X\backslash Z$, the section $s_{y}$ is nonvanishing over each fiber $X_{y}$ for $y\in X\backslash Z$. Since $X_{y}$ are integral, $s_{y}$ is injective, and because $Y$ is flat over $X$, this implies that the vanishing locus of $s$ on $Y$ defines a relative effective Cartier divisor over $X\backslash Z$. This proves $(1)$. By Proposition~\ref{moduli interpretation of bl}, the section $O_{X}\rightarrow f_{*}(L)$ extends to a bundle morphism $O(E)\rightarrow f'_{*}(g^{' *}(L))$ on $Bl_{Z}X$. Hence we get $O(E)\xrightarrow{s'} g^{' *}(L)$ on $Y'$. By our assumption on $L$, this morphism is nonvanishing over each fiber, hence the vanishing locus of $s'$ is a relative effective Cartier divisor over $Bl_{Z}X$.

\end{proof}

The following lemma will be useful when we compute cohomology of sheaves on blowups:
\begin{lemma}\label{pushforward of bl}
Keep the same assumption as above. Consider:
$$
\xymatrix{
  Bl_{Z}X \ar[rr]^{i} \ar[dr]^{p}
                &  &    \textbf{P}(\mathcal{E})=\textrm{Proj}(\textrm{Sym}_{X}\check{\mathcal{E}}) \ar[dl]^{\pi}    \\
                & X
}
$$
Suppose $K$ is an object in $D^{b}_{coh}(\mathbf{P}(\mathcal{E}))$ such that $R\pi_{*}(K\otimes O(a))\in D^{\leq 0}_{coh}(X)$ for all $a\geq 0$, then we have $Rp_{*}Li^{*}K\in D^{\leq 0}_{coh}(X)$

\end{lemma}

\begin{proof}
We are going to show that $R\pi_{*}(K\otimes\bigwedge^{k}\Omega_{\mathbf{P}(\mathcal{E})|X}(1))\in D^{\leq k}(X)$. Since we have a Koszul resolution $\bigwedge\Omega_{\mathbf{P}(\mathcal{E})|X}(1)\rightarrow O_{\bl}$, a spectral sequence argument gives the result. Hence we need to show $R\pi_{*}(K\otimes\bigwedge^{k}\Omega_{\mathbf{P}(\mathcal{E})|X}(1))\in D^{\leq k}(X)$. We do it by induction on $k$. The case when $k=0$ follows from our assumption on $K$. Assume the result holds for all $i$ such that $0\leq i<k$. The Euler sequence on $\mathbf{P}(\mathcal{E})$ takes the form
$$0\rightarrow\Omega_{\mathbf{P}(\mathcal{E})|X}(1)\rightarrow\pi^{*}\check{\mathcal{E}}\rightarrow O(1)\rightarrow 0$$
This give the exact sequence:
$$0\rightarrow\bigwedge^{k}\Omega_{\mathbf{P}(\mathcal{E})|X}(1)\rightarrow \bigwedge^{k}\pi^{*}\check{\mathcal{E}}\rightarrow \bigwedge^{k-1}\Omega_{\mathbf{P}(\mathcal{E})|X}(1)\otimes O(1)\rightarrow 0$$
Hence we get an exact triangle:
$$\bigwedge^{k}\Omega_{\mathbf{P}(\mathcal{E})|X}(1)\otimes K\rightarrow \bigwedge^{k}\pi^{*}\check{\mathcal{E}}\otimes K\rightarrow \bigwedge^{k-1}\Omega_{\mathbf{P}(\mathcal{E})|X}(1)\otimes K\otimes O(1)\rightarrow$$
Notice that $K\otimes O(1)$ satisfies the same assumption as $K$, hence we have 
$$R\pi_{*}(\bigwedge^{k-1}\Omega_{\mathbf{P}(\mathcal{E})|X}(1)\otimes K\otimes O(1))\in D^{\leq k-1}(X)$$ 
by our induction assumption. Now the result follows easily from this.

\end{proof}

The following property results from Grothendieck duality:
\begin{Proposition}\label{dualizing sheaf of bl}
Keep the same assumption as the beginning of this subsection. Let $E$ be the exceptional divisor in $\blA$. The dualizing sheaf $\omega_{\blA\mid A}\simeq O(nE)$.
\end{Proposition}

\subsection{Hilbert scheme of points}
In this section we review some facts about Hilbert scheme of points of curves and surfaces that will be used in the proof of the main theorem. First let us look at Hilbert schemes of $\p$. It is well known that $\Pn$ is the Hilbert scheme of $\p$ parameterizing finite subschemes of length $n$, and $(\p)^{n}=\p\times\cdots\times\p$ classifies finite subscheme of length $n$ together with a flag:
$${\O}=D_{0}\subseteq D_{1}\subseteq\cdots\subseteq D_{n}=D$$
such that $\textrm{ker}(O_{D_{i}}\rightarrow O_{D_{i-1}})$ has length $1$. The natural morphism:
$$(\p)^{n}\xrightarrow{\eta} \Pn \qquad ({\O}=D_{0}\subseteq D_{1}\subseteq\cdots\subseteq D_{n}=D)\rightarrow D$$
is characterized by $\eta^{*}O_{\Pn}(1)\simeq O_{\p}(1)^{\boxtimes n}$. Denote the universal length $n$ subscheme on $\p\times\Pn$ by $D$. Consider:
\[\begin{CD}
\p\times\Pn\\
@V{p}VV\\
\Pn\\
\end{CD}\]
On $\p\times\Pn$ we have:
$$0\rightarrow O_{\p}(-n)\boxtimes O_{\Pn}(-1)\rightarrow O_{\p\times\Pn}\rightarrow O_{D}\rightarrow 0$$
The following lemma is clear:
\begin{lemma}\label{det}
$\textrm{det}(p_{*}O_{D})^{-1}\simeq O(n-1)$, the dualizing sheaf $\omega_{(\p)^{n}\mid \Pn}$ is isomorphic to $O_{\p}(n-1)^{\boxtimes n}$.
\end{lemma}

Now let us review some facts about Hilbert scheme of surfaces and planar curves. The following theorem is well-known (see $[10]$ and $[16]$):
\begin{theorem}
Let $S$ be a smooth surface, then $\HS^{n}$ is smooth of dimension $2n$
\end{theorem}

We also set $\HS^{' n}$ be the open subscheme of $\HS^{n}$ parameterizing subschemes $D$ that can be embedded into a smooth curve (This notion is introduced in $[1]$). In the rest of the paper, $S$ is going to be the total space of $O(n)$ on $\p$, and we shall work with a particular open subscheme of $\HS^{n}$ defined by the following proposition:

\begin{Proposition}\label{an open subscheme of HS}
View $\Pn$ as the Hilbert scheme of $\p$ parameterizing length $n$ subschemes. Let $D$ be the universal subscheme:
$$\xymatrix{
D \ar[rr] \ar[dr] & & \p\times\Pn \ar[ld]^{\pi}\\
 & \Pn\\
}
$$
Then the total space of the vector bundle $\pi_{*}O_{D}(n)$ can be naturally identified with an open subscheme of $\HS^{n}$. Moreover, it is contained in the open subscheme $\HS^{' n}$.

\end{Proposition}

\begin{proof}
Observe that if $D$ is length $n$ subscheme of $\p$ and $s$ a global section of $O_{D}(n)$, then we can embed $D$ into $S$ using the section $s$. This gives $D$ the structure of a closed subscheme of $S$. Now the assertion follows immediately from this.

\end{proof}

As a corollary, we get the following:
\begin{corollary}\label{Hilbert scheme of P1 and S}
Let $V$ denote the scheme corresponding to the affine space $H^{0}(O(n))$.
\begin{enumerate1}
\item There exist a smooth morphism $\Pn\times V\xrightarrow{v} \HS^{n}$ with image contained in $\HS^{' n}$.
\item There is a Cartesian diagram:
$$\xymatrix{
V\times(\p)^{n} \ar[r]^{w} \ar[d]^{\eta} & \Fl^{' n} \ar[d]^{\psi}\\
V\times\Pn \ar[r]^{v} & \HS^{' n}
}
$$
\item Let $D'$ be the universal subscheme of $S$ over $\HS^{' n}$ and $D$ be the universal subscheme of $\p$ over $\Pn$. Consider $O_{D'}$ and $O_{D}$ as vector bundles on $\HS^{' n}$ and $\Pn$. Then we have: $v^{*}\de(O_{D'})\simeq \de(O_{D})$
\end{enumerate1}
\end{corollary}

\begin{proof}
$v$ is defined as follows. A point of $\Pn\times V$ corresponds to pairs $(D,s)$ where $D$ is a closed subscheme of length $n$ of $\p$ and $s\in H^{0}(O(n))$. Since the surface $S$ is the total space of the line bundle $O(n)$, we can embed $D$ into $S$ using the section $s$. So this defines $v$. Now let us consider:
\[\begin{CD}
\p\times\Pn\\
@V{p}VV\\
\Pn\\
\end{CD}\]
Let $D$ be the universal subscheme on $\p\times\Pn$. Let $E$ be the total space of the vector bundle $p_{*}O_{D}(n)$ over $\Pn$. In Proposition~\ref{an open subscheme of HS} we proved that $E$ can be identified with an open subscheme of $\HS^{n}$ contained in $\HS^{' n}$. The natural morphism $O_{\p}\rightarrow O_{D}$ induces a surjection $O(n)\rightarrow O_{D}(n)$. Since $D$ has length $n$, the natural morphism of vector bundles
$$p_{*}O_{\p}(n)\simeq H^{0}(O(n))\otimes O_{\Pn}\rightarrow p_{*}O_{D}(n)$$
is surjective, hence $\Pn\times V\rightarrow E$ is smooth.\\
For part $(2)$, notice that in our situation, the subscheme $D$ of $S$ comes from a subscheme of $\p$, hence giving a flag of $D$ as a subscheme of $S$ is the same thing as giving a flag of $D$ as a subscheme of $\p$, and we know that $\textrm{Flag}^{n}_{\p}\simeq (\p)^{n}$.\\
Part $(3)$ follows almost immediately from the description of $v$. Namely, if $U$ is the open subscheme of $\HS^{' n}$ corresponding to the image of $v$, then the subscheme $D'$ of $S$ comes from the subscheme $D$ of $\p$. Hence the assertion follows.

\end{proof}

Later on we shall also need the following proposition describing the Hilbert scheme of points of planar curves (see $[8]$):

\begin{Proposition}\label{Hilbert scheme of planar curves}
Let
$$\xymatrix{
\mathcal{C} \ar[dr] \ar[rr] & & S\times H \ar[ld]\\
 & H\\
}
$$
be a family of planar curves over $H$ sitting in $S$ which is a relative effective Cartier divisor on $S$. Then the relative Hilbert scheme of $S\times H$ over $H$ is given by $\HS^{n}\times H$ and $\textrm{Hilb}^{n}_{\mathcal{C}|H}$ is a family of locally complete intersection subscheme of codimension $n$ in $\HS^{n}\times H$ over $H$. More precisely, consider the universal subscheme $D$:
$$\xymatrix{
D \ar[rr]^{i} \ar[rd] & & S\times\HS^{n}\times H \ar[ld]^{\pi}\\
 & \HS^{n}\times H\\
}
$$
We have natural global section $s$ of the vector bundle $\pi_{*}O_{D}(\mathcal{C})$ defined by:
$$O_{S}\rightarrow O_{S}(\mathcal{C})\rightarrow O_{D}(\mathcal{C})$$
Then $\textrm{Hilb}^{n}_{\mathcal{C}|H}$ is the vanishing locus of $s$.

\end{Proposition}

\section{Moduli of rank two Higgs bundles on $\p$}

\subsection{Generalities on the geometry of moduli of Higgs bundles}
In this subsection we review some general definitions and results about the geometry of the stack of Higgs bundles on $\p$. We fix an integer $n\geq 2$ throughout the entire paper. Let $\li=O(n)$ on $\p$.

\begin{definition}
A rank $2$ $\li$-valued Higgs bundle on $\p$ is a pair $(E,\phi)$ where $E$ is a rank $2$ vector bundle and the Higgs field $\phi$ is a morphism of vector bundles $E\xrightarrow{\phi} E\otimes\li$

\end{definition}

We denote the stack of Higgs bundles by $\hig$. It decomposes as a disjoint union:
$$\hig=\coprod_{k}\hig^{k}$$
where $\hig^{k}$ is the stack of Higgs bundles of degree $k$. It is well-known that $\hig$ is an algebraic stack locally of finite presentation.
\begin{definition}
A Higgs bundle $(E,\phi)$ is called semistable if the following condition hold: For any line subbundle $E_{0}\subseteq E$ that is preserved by the Higgs field $\phi$ in the sense that $\phi(E_{0})\subseteq E_{0}\otimes\li$, we have $\textrm{deg}(E_{0})\leq \dfrac{\textrm{deg}(E)}{2}$.

\end{definition}
Semistable Higgs bundles form an open substack, denote it by $\hig_{ss}$.(For the construction of the moduli space of semistable higgs bundles, see $[13]$)

\begin{Proposition}\label{smoothness of semistable higgs bundles}
$\hig_{ss}$ is smooth of dimension $4n$.
\end{Proposition}

\begin{proof}
It is well known that the deformation of $\hig$ is controlled by the complex(See $[13]$)
$$K=\mathcal{E}nd(E)\xrightarrow{ad\phi}\mathcal{E}nd(E)\otimes\li$$
Now if $(E,\phi)$ is semistable, and $E\simeq O(a)\oplus O(b)$ with $a\leq b$, then it is not hard to see that we must have $b-a\leq n$, otherwise $O(b)$ would be a subbundle preserved by $\phi$, contrary to the definition of semistability. This implies that
$$H^{1}(\mathcal{E}nd(E)\otimes\li)=0$$
Hence we have $H^{2}(K)=0$, so this implies $\hig_{ss}$ is smooth, and an Euler characteristic computation shows that its dimension is $4n$.

\end{proof}

\begin{definition}\label{generically regular Higgs field}
A Higgs field $\phi$ is called generically regular if $\phi$ is a regular element in $\mathfrak{gl}_{2}(k(\eta))$ after trivilizing $E$ and $\li$ at the generic point $\eta$ of $\p$. Higgs bundles with generically regular Higgs field forms an open substack of $\hig$, denote it by $\widetilde{\hig}$
\end{definition}

\begin{definition}\label{O plus O}
Set $\hig^{(k)}$ to be the stack of Higgs bundles such that the underlying vector bundle $E$ is isomorphic to $O(k)\oplus O(k)$. It is easy to see that this is an open substack of $\hig_{ss}$

\end{definition}

We need the following lemma:
\begin{lemma}\label{lemma in the introduction}
$\hig_{ss}=(\widetilde{\hig}\bigcap\hig_{ss})\bigcup(\bigcup_{k}\hig^{(k)})$.

\end{lemma}

\begin{proof}
Let $(E,\phi)$ be a semistable Higgs bundle that is not in $\bigcup_{k}\hig^{(k)}$. Then we can write $E\simeq O(a)\oplus O(b)$ with $a<b$. Then by the definition of semistability, $O(b)$ cannot be preserved by $\phi$, hence $\phi$ is regular at the generic point of $\p$, so $(E,\phi)\in\widetilde{\hig}\bigcap\hig_{ss}$.

\end{proof}

\subsection{Hitchin fibration and spectral curve}
In this subsection we are going to review certain properties of the Hitchin fibration that will be used later in the paper. Part $(2)$ of Proposition~\ref{properties of Hitchin fibration} will be frequently used. Lemma~\ref{resolution} will be used in Section $4$ to construct resolutions of sheaves. Most of the properties are well-known(See $[12]$ and the appendix of $[3]$ for a summary).

\begin{definition}
The Hitchin base $H$ is the affine space given by:
$$H^{0}(\p,\li)\times H^{0}(\p,\li^{2})$$

\end{definition}

It is well-known that we have the Hitchin fibration:
$$\hig\xrightarrow{h} H \qquad (E,\phi)\rightarrow (\textrm{tr}(\phi),\textrm{det}(\phi))$$
The following proposition summarizes the properties of Hitchin fibration:
\begin{Proposition}\label{properties of Hitchin fibration}
\begin{enumerate1}
\item $H$ is an affine space of dimension $3n+2$
\item $h$ is a relative complete intersection morphism with fiber dimension $n-2$
\end{enumerate1}

\end{Proposition}

The proof will be given at the end of this subsection.

Let $S$ be the surface defined by the total space of the line bundle $\li$. We have:
$$\xymatrix{
S=\textrm{Spec}(\textrm{Sym}(\li^{-1})) \ar[d]^{q}\\
\p\\
}
$$
Any Higgs bundle can be naturally viewed as a coherent sheaf on the surface $S$, and we have the following lemma:

\begin{lemma}\label{resolution}
Let $(E,\phi)$ be a Higgs bundle. There is a locally free resolution of $E$ as a coherent sheaf on $S$:
$$0\rightarrow q^{*}(E)\otimes q^{*}(\li^{-1})\rightarrow q^{*}(E)\rightarrow E\rightarrow 0$$
\end{lemma}

\begin{proof}
By the definition of $S$, there is a tautological section $T$ of $q^{*}(\li)$ on $S$, and the morphism
$$q^{*}(E)\otimes q^{*}(\li^{-1})\rightarrow q^{*}(E)$$
is given by $T-q^{*}(\phi)$. It is easy to see that this is a resolution of $E$ as an $O_{S}$ module.

\end{proof}

Given a Higgs bundle $(E,\phi)$ we can consider the characteristic polynomial of $\phi$:
$$\textrm{det}(T-\phi)=T^{2}-\textrm{tr}(\phi)T+\textrm{det}(\phi)$$
which is naturally a global section of $q^{*}(\li^{2})$ on $S$. Its zero locus cuts out a curve $C$ which is a finite cover of degree $2$ over $\p$, called the spectral curve of $(E,\phi)$. By construction, $E$ is naturally a coherent sheaf on its spectral curve $C$. We have a universal family of spectral curves $\mathcal{C}$ over the Hitchin base $H$:
$$\xymatrix{
\mathcal{C} \ar[rr] \ar[rd] & & S\times H \ar[ld]\\
& H\\
}
$$
The following proposition is well-known(See $[12]$).
\begin{Proposition}\label{spectral data}
Let $(E,\phi)$ be a Higgs bundle and $C$ be the corresponding spectral curve.
\begin{enumerate1}
\item $C\subseteq S$ is an effective divisor given by the vanishing locus of a section of $q^{*}(\li)^{2}$.
\item $C\rightarrow \p$ is a degree $2$ flat cover and $O_{C}\simeq O\oplus\li^{-1}$ as $O_{\p}$ module.
\item $E$ is a torsion-free rank one sheaf on $C$.
\item There is a bijection between Higgs bundles with generically regular Higgs field and maximal Cohen-Macaualy sheaves on $C$ that are generically line bundles.
\item If the spectral curve $C$ is reduced, then $E$ is automatically generically rank one as a coherent sheaf on $C$, hence $(E,\phi)\in\widetilde{\hig}$
\item If the spectral curve $C$ is nonreduced and $\phi$ is not generically regular, then $E$ is supported on $C_{red}\simeq \p$ and there exist a global section $\lambda\in H^{0}(\p,\li)$ such that $C$ is cut out by $(T-\lambda)^{2}$ in $S$ and $\phi=id_{E}\otimes\lambda$.
\end{enumerate1}
\end{Proposition}

The following simple lemma will be used later
\begin{lemma}\label{description of the dual}
Let $(E,\phi)$ be a rank $2$ Higgs bundle with value in $O(n)$, considered as a coherent sheaf on the spectral curve $C$. Then we have the following isomorphism of vector bundles on $\p$:
$$\mathcal{H}om_{O_{C}}(E,O_{C})\simeq \mathcal{H}om_{O_{\p}}(E,O_{\p})\otimes O(-n)$$
\end{lemma}
\begin{proof}
Let $F$ be the coherent sheaf on $C$ corresponding to $(E,\phi)$. Consider:
$$\xymatrix{
C \ar[rd]^{\pi} \ar[rr]^{i} & & S \ar[ld]^{q}\\
 & \p\\
}
$$
Since $S$ is the total space of the line bundle $O(n)$, and $C$ is the vanishing locus of a section of the line bundle $q^{*}(O(2n))$ on $S$, we have that the relative dualizing complex of $\pi$ can be identified with $\pi^{*}(O(n))$. Hence by Grothendieck duality, we have:
$$\mathcal{H}om(\pi_{*}(F),O_{\p})\simeq \mathcal{H}om_{O_{C}}(F,\pi^{*}(O(n)))\simeq\mathcal{H}om_{O_{C}}(F,O_{C})\otimes O(n)$$
The claim now follows from this.
\end{proof}

In order to establish part $(2)$ of Propositon~\ref{properties of Hitchin fibration}, we need the following lemma. It is proved by using part $(4)$ of the Proposition~\ref{spectral data}:
\begin{lemma}
$\widetilde{\hig}\xrightarrow{h} H$ is a relative complete intersection morphism with fiber dimension $n-2$
\end{lemma}

\begin{proof}
Let $\mathcal{C}$ be the universal spectral curve over $H$. Also let $L$ the line bundle on $\mathcal{C}$ which is the pullback of $O(1)$ from $\p$. So $L$ is relatively ample. We have the Abel-Jacobian map defined by:
$$\textrm{Hilb}^{N}_{\mathcal{C}|H}\xrightarrow{\alpha}\widetilde{\hig}^{d} \qquad D\rightarrow \check{I}_{D}\otimes L^{-m}$$
where $N=d+2(m-1)+n$. Let $U_{N}$ be the open subscheme of $\textrm{Hilb}^{N}_{\mathcal{C}|H}$ determined by the condition $H^{1}(\check{I}_{D})=0$. Then it is not hard to check that the restriction of $\alpha$ to $U_{N}$ is smooth of relative dimension $N+2-n$. Moreover, if $(E,\phi)\in\widetilde{\hig}^{d}$, then view $E$ as a coherent sheaf on the corresponding spectral curve $C$, we can find $m>>0$ such that $H^{1}(E\otimes L^{m})=0$, and $E\otimes L^{m}$ is globally generated. Now since the Higgs field is generically regular, we can choose a global section $s$ such that $O_{C}\xrightarrow{s} E\otimes L^{m}$ is injective. So setting $I=\mathcal{H}om_{O_{C}}(E\otimes L^{m},O_{C})$, $s$ induces
$$0\rightarrow I\rightarrow O_{C}\rightarrow O_{D}\rightarrow 0$$
Hence we have
$$\alpha(D)=(E,\phi)$$
This implies that any $(E,\phi)\in\widetilde{\hig}$ is in the image of $U_{N}$ for some $N>>0$. Now the claim follows from Proposition~\ref{Hilbert scheme of planar curves} and the fact that $\alpha$ is smooth of relative dimension $N+2-n$.

\end{proof}

Now we can prove Proposition~\ref{properties of Hitchin fibration}
\begin{proof}
Part $(1)$ is immediate from the definition of $H$. For part $(2)$, it is known that under our assumptions $\hig$ is a locally complete intersection stack of pure dimension $4n$ ($[5]$ Proposition $3.2$). Hence we only need to show the fibers of $h$ has dimension less than or equals to $n-2$ by Proposition~\ref{criterion for relative complete intersection}. The previous lemma implies that we only need to study the complement of $\widetilde{\hig}$ in each fiber. Also by part $(5)$ of Proposition~\ref{spectral data}, we only need to look at nonreduced spectral curves. By part $(6)$ of Proposition~\ref{spectral data}, the complement of $\widetilde{\hig}$ can be identified with the stack of rank $2$ vector bundles on $\p$, hence it has dimension less than $n-2$, this finishes the proof.

\end{proof}

\subsection{Higgs bundles and the Hilbert scheme}

%define the blowup and study the geometry of the morphism to hilbert scheme%

In this subsection we are going to define certain important objects mentioned in subsection~\ref{outline} and study their geometric properties. The main result is Proposition~\ref{resolve the rational map}, which says that $\hig''$ resolves the rational map:
$$\hig'\dashrightarrow \textrm{Hilb}^{n}_{\mathcal{C}|H}$$
(See the notations in Proposition~\ref{resolve the rational map}). This property is central to the construction of the Poincar\'e sheaf (See the diagram in Theorem~\ref{main theorem}). First let us define all the relevant objects.
\begin{definition}
Let $\hig'$ be the moduli stack classifying the data $(E,\phi,s)$ where $(E,\phi)$ is a rank $2$ Higgs bundle on $\p$ with value in $\li$ such that $E\simeq O\oplus O$, $s$ is a nonzero global section of $E$.
\end{definition}
Notice that in this case $O_{\p}$ is a subbundle of $E$ via the section $s$. There is a natural projection $\hig'\rightarrow \hig^{(0)}$ (Recall our notation in Definition~\ref{O plus O}) given by:
$$(E,\phi,s)\rightarrow (E,\phi)$$

The following proposition is obvious:
\begin{Proposition}\label{propositions of hig'}
\begin{enumerate1}
\item $\hig'$ is smooth of dimension $4n+2$
\item The natural morphism $\hig'\rightarrow\hig^{(0)}$ is smooth with fibers isomorphic to $\mathbb{A}^{2}\backslash \{0\}$.
\end{enumerate1}
\end{Proposition}

There is a more explicit model for $\hig'$. Namely, consider the scheme $A$ defined by
$$A=\left(\begin{array}{ll}
x &y\\
z &w\\
\end{array}\right)
$$
where
$$x, y, z, w\in H^{0}(\p,O(n))$$
On $\p\times A$ we have $O\oplus O\xrightarrow{\phi} \li\oplus\li$ where $\phi$ is given by:
$$\phi=\left(\begin{array}{ll}
x &y\\
z &w\\
\end{array}\right)
$$
Also we have a nonvanishing global section $O\rightarrow O\oplus O$ given by $1\rightarrow (1,0)$. Hence we get a morphism: $A\rightarrow\hig'$.

The following lemma is immediate from the definition:
\begin{lemma}
We have a natural smooth morphism $A\rightarrow \hig'$.
\end{lemma}

Next let us consider:
$$\xymatrix{
\p\times\hig' \ar[d]^{f}\\
\hig'\\
}
$$

By the definition of $\hig'$, we have a bundle map: $O\xrightarrow{s} E$ on $\p\times\hig'$, denote the quotient bundle $E/O$ by $N$. We have the following morphisms of coherent sheaves on $\p\times\hig'$:
$$O\xrightarrow{s} E\xrightarrow{\phi} E\otimes\li\rightarrow N\otimes\li$$
Hence we get:
$$O\rightarrow f_{*}(N\otimes\li)$$
By the definition of $\hig'$, $f_{*}(N\otimes\li)$ is a vector bundle of rank $n+1$ on $\hig'$.
\begin{lemma}\label{definition of Z}
Let $\mathcal{Z}$ be the vanishing locus of the morphism $O\rightarrow f_{*}(N\otimes\li)$ on $\hig'$. Then $\mathcal{Z}$ is locally a complete intersection of codimension $n+1$ in $\hig'$ and $\mathcal{Z}$ is smooth.
\end{lemma}

\begin{proof}
It is enough to prove the statement on $A$. It is easy to check that the pullback of $\mathcal{Z}$ to $A$ is the closed subscheme given by $z=0$. So the claim follows easily from this.

\end{proof}

If we view $E$ as a coherent sheaf on the spectral curve $C$, then the section $s$ of $E$ can also be viewed as a global section of $E$ as coherent sheaf on $C$, hence $s$ induces the following morphism on $\p\times\hig'$:
$$O\oplus\li^{-1}\simeq O_{C}\xrightarrow{s} E$$
Also if we take the dual of this, we get:
$$\mathcal{H}om_{O_{C}}(E,O_{C})\rightarrow O_{C}\rightarrow O_{D'}\rightarrow 0$$
Observe that if we pullback this morphism to $\p\times A$, then the matrix representation of $s$ is given by:
$$s=\left(\begin{array}{ll}
1 &x\\
0 &z\\
\end{array}\right)
$$
Hence $s$ is injective on the complement of $\mathcal{Z}$, so $D'$ is a family of subscheme of length $n$ of $C$ over $\hig'\backslash\mathcal{Z}$. Hence we have the following:
\begin{lemma}\label{rational map to HS}
There are natural morphisms
$$\hig'\backslash\mathcal{Z}\rightarrow \textrm{Hilb}^{n}_{\mathcal{C}|H}\rightarrow\HS^{n}\times H\rightarrow\HS^{n}$$
\end{lemma}

The following Proposition is central to the construction of the Poincar\'e sheaf:
\begin{Proposition}\label{resolve the rational map}
Let $\hig''$ be the blowup of $\hig'$ along the closed substack $\mathcal{Z}$. Then the morphism $\hig'\backslash\mathcal{Z}\rightarrow \textrm{Hilb}^{n}_{\mathcal{C}|H}$ extends to a morphism of stacks over $H$:
$\hig''\rightarrow \textrm{Hilb}^{n}_{\mathcal{C}|H}$. Hence we have the following diagram:
$$\xymatrix{
\hig'' \ar[d]^{\pi} \ar[r] & \HuC \ar[r] & \HS^{n}\\
\hig'
}
$$
\end{Proposition}

\begin{proof}
First let us give another description for the morphism $$\hig'\backslash\mathcal{Z}\rightarrow\textrm{Hilb}^{n}_{\mathcal{C}|H}$$
By the definition of $\hig'$, we have:
$$O\xrightarrow{s} E\xrightarrow{\phi} E\otimes\li\rightarrow N\otimes\li$$
on $\p\times\hig'$. Denote the vanishing locus of the resulting morphism $O\rightarrow N\otimes\li$ by $D'$. From its definition, if we restrict $O\xrightarrow{s} E$ to $D'$, it factors uniquely as:
$$\xymatrix{
O_{D'} \ar[r] \ar[d] & O_{D'}\otimes\li \ar[d]\\
E_{D'} \ar[r]^{\phi_{D'}} & E_{D'}\otimes\li_{D'}\\
}
$$
Hence we get a global section of $\li_{D'}$, and we can embed $D'$ into the surface $S$ using this section. So in this way, we get a morphism:
$$\hig'\backslash\mathcal{Z}\rightarrow \HS^{n}$$
which coincide with the morphism defined in Lemma~\ref{rational map to HS}:
$$\hig'\backslash\mathcal{Z}\rightarrow\textrm{Hilb}^{n}_{\mathcal{C}|H}\rightarrow\HS^{n}$$
Next let us consider:
$$\xymatrix{
\p\times\hig'' \ar[r]^{\pi'} \ar[d]^{f'} & \p\times\hig' \ar[d]^{f}\\
\hig'' \ar[r]^{\pi} & \hig'\\
}
$$
By Corollary~\ref{bl resolve nonflatness}, the section:
$$O\rightarrow N\otimes\li$$
extends to:
$$O\rightarrow O(E)\rightarrow \pi^{'*}(N\otimes\li)$$
on $\hig''$ where $E$ is the exceptional divisor, and the vanishing locus of
$$O(E)\rightarrow N\otimes\li$$
on $\p\times\hig''$ defines a family of closed subscheme of length $n$ of $\p$ over $\hig''$. Denote it by $D$. From the construction of $D$, it is clear that $D$ is a subscheme of the pullback of $D'$ to $\p\times\hig''$ and when restricted to $\hig''\backslash E$, $D$ coincide with $D'$. Since $D$ is a subscheme of the pullback of $D'$, the global section of $\li_{D'}$ gives a global section of $\li_{D}$, and we embed $D$ into $S$ using this section. Moreover, since $D$ is a subscheme of $D'$ as subschemes of $S$, and $D'$ is a subscheme of the spectral curve $C$, we see that the morphism $\hig''\rightarrow\HS^{n}\times H$ factors through:
$$\hig''\rightarrow\HuC\rightarrow\HS^{n}\times H$$
So the claim follows from this.

\end{proof}

\begin{corollary}
The image of $\hig''\rightarrow\HS^{n}$ lies in $\HS^{' n}$. More precisely, the image lies in the open subscheme of $\HS$ described in Proposition~\ref{an open subscheme of HS}
\end{corollary}

\begin{proof}
This follows from the description about $D$ in Proposition~\ref{resolve the rational map}. Namely, $D$ comes from the embedding of a closed subscheme $D$ of $\p$ into $S$ via a global section of $\li_{D}$.
\end{proof}

It is also convenient to describe the morphism $\hig''\rightarrow\HuC$ more explicitly in terms of $A$ and its blowups, and this is provided by the following:

\begin{corollary}\label{explicit description of bl resolve rational map}
Let $Z$ be the closed subscheme of $A$ given by $z=0$. Then $\blA$ is a smooth cover of $\hig''$ and there is a morphism of schemes $\blA\xrightarrow{p} \HuC$ which is induced by $\hig''\rightarrow\HuC$. Moreover, $\blA$ and $\hig''$ are smooth, and $\hig''$ is a local complete intersection over $H$
\end{corollary}
\begin{proof}
The first assertion comes from the fact that $A$ is smooth over $\hig'$, and the pullback of $\mathcal{Z}$ to $A$ is exactly $Z$. The second assertion follows from the fact that $\blA$ is smooth. Also, because $A$ is a local complete intersection over $H$ (Proposition~\ref{properties of Hitchin fibration}), and $\blA$ is a local complete intersection over $A$ (Proposition~\ref{blowup as regular embedding}), we conclude that $\hig''$ is a local complete intersection over $H$.
\end{proof}

Now let us give an explicit description of the morphism $\blA\xrightarrow{p}\textrm{Hilb}^{n}_{\mathcal{C}|H}$. Recall that $A$ is the affine space given by:
$$A=\left(\begin{array}{ll}
x &y\\
z &w\\
\end{array}\right)
$$
where
$$x, y, z, w\in H^{0}(\p,O(n))$$
and the pullback of $\mathcal{Z}$ to $A$ is the subscheme $z=0$, denote it by $Z$. Since $A\rightarrow\hig'$ is smooth, we also have a smooth morphism $\blA\rightarrow\hig''$. Let $A^{*}$ be the open subscheme defined by $z\neq 0$. On $\p\times A$ we have a global section $s$ of $O\oplus O$ given by $1\xrightarrow{s}(1,0)$. The structure sheaf of the spectral curve is given by $O_{C}=O\oplus O(-n)$, and if we view $O\oplus O$ as a sheaf on the spectral curve $C$, then $s$ induces a morphism:
$$O_{C}\simeq O\oplus O(-n)\xrightarrow{s} O\oplus O$$
$$s=\left(\begin{array}{ll}
1&x\\
0&z
\end{array}\right)
$$
So we get an injection $O_{C}\xrightarrow{s} E$. When $z\neq 0$, $s$ defines morphisms:
$$\xymatrix{
A^{*} \ar[rr]^{\alpha} \ar[rd]^{\gamma} &  & \textrm{Hilb}_{\mathcal{C}\mid H}^{n} \ar[ld]\\
  & \textrm{Hilb}_{S}^{n}\\
}
$$
The morphism $\gamma$ can be described more explicitly in the following way. First, on $A^{*}$, $z=0$ defines a finite subscheme of length $n$ on $\p$, and we embed it into the surface $S$ using the section $x\in H^{0}(O(n))$. So if we think of $\Pn$ as the Hilbert scheme of $\p$, then the morphism $\gamma$ can be factored as :
$$A^{*}\rightarrow \Pn\times V\rightarrow \textrm{Hilb}_{S}^{n}$$
$$(x,y,z,w)\rightarrow ([z],x)\rightarrow D$$
where $V$ is the affine space $H^{0}(O(n))$, and $[z]$ is the closed subscheme of $\p$ determined by $z=0$. The map $A^{*}\rightarrow\Pn$ is undefined at $z=0$, but we can resolve it by taking the blowup of $A$ along $Z$, hence we get the following picture:
\begin{equation}\label{blowup}\xymatrix{
  & \blA \ar[ld] \ar[rd]\\
A &   & \textrm{Hilb}_{\mathcal{C}\mid H}^{n} \ar[r] & \textrm{Hilb}_{S}^{n}\\
}
\end{equation}

As a byproduct of the explicit description above, we get the following:
\begin{lemma}\label{factorization to Pn}
The morphism $\blA\rightarrow\HS^{n}$ factors as:
$$\blA\rightarrow\Pn\times A\xrightarrow{v}\HS^{n}$$
Also $\Pn\times A\xrightarrow{v}\HS^{n}$ is smooth.
\end{lemma}

\begin{proof}
First notice that $\blA$ naturally embeds into $\Pn\times A$. As we have already seen in the previous proof, there is natural morphism:
$$A^{*}\rightarrow \Pn\times V\rightarrow \HS^{n}$$
$$(x,y,z,w)\rightarrow ([z],x)\rightarrow D$$
where $V$ is the affine space $H^{0}(O(n))$ and $[z]$ is the finite subscheme of $\p$ cut out by $z=0$. This extends to a morphism:
$$\blA\rightarrow \Pn\times A\rightarrow\Pn\times V\rightarrow \HS^{n}$$
where $\Pn\times A\rightarrow\Pn\times V$ is induced by the projection $A\rightarrow V$:
$$(x,y,z,w)\rightarrow x$$
In Proposition~\ref{Hilbert scheme of P1 and S} we showed that $\Pn\times V\rightarrow\HS^{n}$ is smooth, so the claim follows from this.

\end{proof}

Since we have a morphism $\blA\rightarrow\HuC$, we have a closed subscheme of length $n$ of $C$ on $\p\times\blA$. Hence we have an exact sequence:
$$0\rightarrow K\rightarrow O_{C}\rightarrow O_{D}\rightarrow 0$$
For later use, let us give a description of the kernel $K$:
\begin{lemma}\label{description of kernel}
Consider:
\[\begin{CD}
\p\times\blA\\
@V{f}VV\\
\blA\\
\end{CD}\]
Then $K$ fits into the following exact sequence:
$$0\rightarrow O(-n)\otimes O(E)\rightarrow K\rightarrow O(-n)\rightarrow 0$$
where $E$ is the exceptional divisor on $\blA$.

\end{lemma}

\begin{proof}
By our discussions in Corollary~\ref{bl resolve nonflatness} and Corollary~\ref{explicit description of bl resolve rational map}, the morphism $O\xrightarrow{z} O(n)$ on $\p\times A$ extends to $O(E)\rightarrow O(n)$ on $\p\times\blA$, and the subscheme $D$ of $C$ is also a subscheme of $\p$ constructed from the vanishing locus of $O(E)\rightarrow O(n)$. Hence we have the following commutative diagram:
$$\xymatrix{
0 \ar[r] & O(E)\otimes O(-n) \ar[r] \ar[d] & O_{\p} \ar[r] \ar[d] & O_{D} \ar[r] \ar[d]^{id} & 0\\
0 \ar[r] & K \ar[r] & O_{C} \ar[r] & O_{D} \ar[r] & 0\\
}
$$
Since $O_{C}\simeq O_{\p}\oplus O_{\p}(-n)$, the assertion follows from the diagram above.

\end{proof}

\subsection{The key diagram and its geometric properties}

As we have already seen in the introduction, the following diagram is central to the construction of the Poincar\'e sheaf (see Theorem~\ref{main theorem}). We shall refer to it as the key diagram:
$$\xymatrix{
\hig''\times_{H}\hig \ar[r] \ar[d] & \textrm{Hilb}^{n}_{\mathcal{C}|H}\times_{H}\hig\hookrightarrow\HS^{n}\times\hig\\
\hig'\times_{H}\hig\\
}
$$
In this section we are going to study some geometric properties of the morphisms in the key diagram. We will establish the following:

\begin{theorem}\label{properties of g}
Consider the morphism $\hig''\rightarrow \HuC$ constructed in Proposition~\ref{resolve the rational map}, which induces: $\hig''\times_{H}\hig\xrightarrow{g}\HuC\times_{H}\hig$. Then $g$ is a local complete intersection morphism.
\end{theorem}

Assume this for the moment, we have the following corollary:
\begin{corollary}\label{part 1 of the main theorem}
Recall we have the maximal Cohen-Macaulay sheaf $Q$ on $\HuC\times_{H}\hig$ (Proposition~\ref{definition of Q}). Then $g^{*}(Q)$ is a maximal Cohen-Macaulay sheaf on $\hig''\times_{H}\hig$.

\end{corollary}

\begin{proof}
By part $(2)$ of Proposition~\ref{properties of lci}, $g$ is of finite tor-dimension. Moreover, since the morphism $\hig\rightarrow H$ is a relative complete intersection (Proposition~\ref{properties of Hitchin fibration}), and $\hig''$ is a local complete intersection over $H$ (Corollary~\ref{explicit description of bl resolve rational map}), by part $(1)$ of Proposition~\ref{properties of lci} we see that $\hig''\times_{H}\hig$ is a relative complete intersection over $H$, hence it is Gorenstein. Now Proposition~\ref{pullback of CM sheaf} finishes the proof.

\end{proof}
To establish the theorem, and also to carry out the calculations in the next section, we need to take a closer look at the morphism $g$. First let us notice the following: There is a natural morphism $\hig''\times\hig\rightarrow H\times H$ where $H$ is the Hitchin base, so over
$$\mathbf{P}^{1}\times \hig''\times\hig$$
we have two families of curves in $S$ corresponding to the two components of $H$. We shall denote $C_{1}$ the pullback of the curve from $\hig''\rightarrow H$ and $C_{2}$ the curve from $\hig\rightarrow H$. Since we have a morphism $\hig''\rightarrow \textrm{Hilb}^{n}_{\mathcal{C}_{1}|H}$ by Proposition~\ref{resolve the rational map}, on $\mathbf{P}^{1}\times\hig''\times\hig$ we have $O_{C_{1}}\rightarrow O_{D}\rightarrow0$ where $D$ is the length $n$ subscheme of $C_{1}$. The following proposition summarizes the properties that we need:

\begin{Proposition}\label{geometric property of p}
Consider the morphism $\hig''\times\hig\xrightarrow{p}\HS^{n}\times\hig$ provided by Proposition\ref{resolve the rational map}, and set $W=\textrm{Supp}(Q)$, then we have:
\begin{enumerate1}
\item $W$ is locally a complete intersection of codimension $n$ in $\hig''\times\hig$
\item $\hig''\times_{H}\hig$ is locally a complete intersection in $W$.
\end{enumerate1}
\end{Proposition}

Note that this Proposition imlies the following:
\begin{corollary}\label{pullback of Q}
$Lp^{*}(Q)\simeq p^{*}(Q)$, and $p^{*}(Q)$ is a maximal Cohen-Macaulay sheaf supported on $W$.
\end{corollary}
\begin{proof}(Of the corollary)
We know that $Q$ is a maximal Cohen-Macaulay sheaf supported on $\textrm{Hilb}^{n}_{\mathcal{C}_{2}|H}\times_{H}\hig$(Proposition~\ref{definition of Q}), which is a regular embedding of codimension $n$ in $\HS^{n}\times\hig$. Since $W$ is the pullback of $\textrm{Hilb}^{n}_{\mathcal{C}_{2}|H}\times_{H}\hig$, Part $(1)$ combined with Proposition~\ref{pullback of CM sheaf} implies the claim of the corollary.
\end{proof}

The proof of Proposition~\ref{geometric property of p} is based on the following lemma:
\begin{lemma}
Consider:
\[\begin{CD}
\p\times\hig''\times\hig\\
@V{f}VV\\
\hig''\times\hig\\
\end{CD}\]
\begin{enumerate1}
\item $\hig''\times_{H}\hig$ is the vanishing locus of a section of the vector bundle $f_{*}O_{C_{1}}(C_{2})$ on $\hig''\times\hig$ and it is locally a complete intersection of codimension $3n+2$.
\item The natural morphism of vector bundles $f_{*}O_{C_{1}}(C_{2})\rightarrow f_{*}O_{D}(C_{2})$ is surjective.
\end{enumerate1}
\end{lemma}

\begin{proof}
For part $(1)$, notice that we have a canonical morphism $O_{S}\rightarrow O_{S}(C_{2})$ on $S\times\hig''\times\hig$. Restricting it to $C_{1}$ we get $O_{C_{1}}\rightarrow O_{C_{1}}(C_{2})$ on $\p\times\hig''\times\hig$. Hence we get a section of $O_{C_{1}}(C_{2})$ which induces a section $t$ of $f_{*}O_{C_{1}}(C_{2})$. Notice that since $O_{C_{1}}\simeq O\oplus O(-n)$, by part $(1)$ of Proposition~\ref{spectral data} we have that
$$O_{C_{1}}(C_{2})\simeq O_{C_{1}}\otimes O(2n)\simeq O(n)\oplus O(2n)$$
Hence $f_{*}O_{C_{1}}(C_{2})$ is a vector bundle of rank $3n+2$ on $\hig''\times\hig$. It follows from the definition that the vanishing locus of the section $t$ of $f_{*}O_{C_{1}}(C_{2})$ can be identified with $\hig''\times_{H}\hig$. We have $\textrm{dim}(\hig'')=\textrm{dim}(\hig')$. Using the fact that the Hitchin fibration $\hig\rightarrow H$ is flat (Proposition~\ref{properties of Hitchin fibration}), $\hig''\times_{H}\hig$ has codimension equals to $\textrm{dim}(H)=3n+2$ in $\hig''\times\hig$. Hence the section $t$ of the vector bundle $f_{*}O_{C_{1}}(C_{2})$ defines a regular sequence of length $3n+2$ on $\hig''\times\hig$.\\
For part $(2)$, since $\blA$ is a smooth cover of $\hig''$, it is enough to prove the statement on $\blA\times\hig$. By Lemma~\ref{description of kernel}, we have the following exact sequences on $\p\times\blA\times\hig$:
$$0\rightarrow K\rightarrow O_{C_{1}}\rightarrow O_{D}\rightarrow 0$$
$$0\rightarrow O(-n)\otimes O(E)\rightarrow K\rightarrow O(-n)\rightarrow 0$$
Since $K(C_{2})\simeq K\otimes O(2n)$, we see that $R^{1}f_{*}(K(C_{2}))=0$, hence the natural morphism:
$$f_{*}O_{C_{1}}(C_{2})\rightarrow f_{*}O_{D}(C_{2})$$
is surjective

\end{proof}

\begin{proof}(Of Proposition~\ref{geometric property of p})
Since $\blA$ is a smooth cover of $\hig''$, we only need to prove the corresponding statement for $\blA\times\hig\xrightarrow{p}\HS^{n}\times\hig$. Let $D$ be the universal subscheme on $S\times\HS^{n}$, and consider
\[\begin{CD}
S\times\HS^{n}\times\hig\\
@V{\pi}VV\\
\HS^{n}\times\hig\\
\end{CD}\]

By Proposition~\ref{definition of Q}, $Q$ is supported on
$$\textrm{Hilb}_{\mathcal{C}_{2}|H}^{n}\times_{H}\hig\subseteq \textrm{Hilb}_{S_{H}|H}^{n}\times_{H}\hig=\HS^{n}\times \hig$$
There is a canonical section $s$ of $\pi_{*}O_{D}(C_{2})$ defined by
$$O_{S}\rightarrow O_{S}(C_{2})\rightarrow O_{D}(C_{2})$$
Also from Proposition~\ref{Hilbert scheme of planar curves} we know that
$$\textrm{Hilb}_{\mathcal{C}_{2}|H}^{n}\times_{H}\hig\subseteq \HS^{n}\times\hig$$
is locally a complete intersection of codimension $n$ corresponds to the vanishing locus of $s$. Consider:
\[\begin{CD}
\p\times\blA\times\hig\\
@V{f}VV\\
\blA\times\hig\\
\end{CD}\]
Since $Q$ is supported on $\textrm{Hilb}^{n}_{\mathcal{C}_{2}|H}\times_{H}\hig$, $W$ is the vanishing locus of the section $p^{*}(s)$ of $f_{*}O_{D}(C_{2})$ on $\blA\times\hig$. Notice that we also have a canonical section $t$ of $f_{*}O_{C_{1}}(C_{2})$ coming from $O_{S}\rightarrow O_{S}(C_{2})$, and the section $s$ of $f_{*}O_{D}(C_{2})$ is the image of $t$ under the morphism
$$f_{*}O_{C_{1}}(C_{2})\rightarrow f_{*}O_{D}(C_{2})$$
In part $(1)$ of the previous lemma we already showed that the section $t$ of $f_{*}O_{C_{1}}(C_{2})$ defines a regular sequence on $\blA\times\hig$. Also part $(2)$ of the previous lemma implies that $f_{*}O_{C_{1}}(C_{2})\rightarrow f_{*}O_{D}(C_{2})$ is surjective. Since $p^{*}(s)$ is the image of $t$, it follows that $p^{*}(s)$ also defines a regular sequence of length $n$ on $\blA\times\hig$. This proves $(1)$\\
For $(2)$, notice that we have an exact sequence of vector bundles:
$$0\rightarrow f_{*}(K(C_{2}))\rightarrow f_{*}(O_{C_{1}}(C_{2}))\rightarrow f_{*}O_{D}(C_{2})\rightarrow 0$$
Since $W$ is the vanishing locus of the image of the section $t$ in $f_{*}(O_{D}(C_{2}))$, if we restrict $t$ to $W$, then it factors through $f_{*}(K(C_{2}))|_{W}$. Since $t$ determines a regular sequence on $\blA\times\hig$, the section $t|_{W}$ of $f_{*}(K(C_{2}))|_{W}$ also gives a regular sequence on $W$, and $\blA\times_{H}\hig$ is the vanishing locus of $t|_{W}$. This proves part $(2)$.

\end{proof}

As a byproduct of the proof, using Lemma~\ref{description of kernel} we get the following:
\begin{lemma}\label{Koszul resolution on W}
Consider the vector bundle $\mathcal{E}=f_{*}(K(C_{2}))$ on $\blA\times\hig$ , then $\blA\times_{H}\hig$ is the vanishing locus of a global section of $\mathcal{E}|_{W}$ in $W$. Also $\mathcal{E}$ fits into an exact sequence:
$$0\rightarrow O(E)^{n+1}\rightarrow\mathcal{E}\rightarrow O^{n+1}\rightarrow 0$$
\end{lemma}

Now we can prove Theorem~\ref{properties of g}:
\begin{proof}
Since $\blA$ is a smooth cover of $\hig''$, $\hig''$ is also smooth, hence by part $(4)$ of Proposition~\ref{properties of lci}, we see that $\hig''\rightarrow\HS^{n}$ is a local complete intersection morphism, and part $(1)$ of Proposition~\ref{properties of lci} implies that $\hig''\times\hig\rightarrow \HS^{n}\times\hig$ is a local complete intersection morphism. In Proposition~\ref{geometric property of p} we proved that $W$ is the pullback of
$$\textrm{Hilb}^{n}_{\mathcal{C}_{2}|H}\times_{H}\hig\subseteq\HS^{n}\times\hig$$
and $W$ is still locally a complete intersection of codimension $n$ in $\hig''\times\hig$, hence part $(5)$ of Proposition~\ref{properties of lci} implies that $W\rightarrow \textrm{Hilb}^{n}_{\mathcal{C}_{2}|H}\times_{H}\hig$ is a local complete intersection morphism. Now the theorem follows from part $(2)$ of Proposition ~\ref{geometric property of p}.

\end{proof}

\section{A cohomological vanishing result}
The purpose of this section is to establish a cohomological vanishing result which will be used in the last section to prove part $(2)$ of the main theorem (Theorem~\ref{main theorem}). The main result is Lemma~\ref{cohomological vanishing}. To state it, let us recall that in Lemma~\ref{factorization to Pn} we have a smooth morphism
$$A\times\Pn\xrightarrow{v} \HS^{n}$$
Also, Corollary~\ref{Hilbert scheme of P1 and S} implies that we have a Cartesian diagram:
\[\begin{CD}
A\times(\p)^{n}@>>>\Fl^{' n}\\
@VVV@VVV\\
A\times\Pn@>{v}>>\HS^{' n}\\
\end{CD}\]
We can now state the main result of this section:
\begin{lemma}\label{cohomological vanishing}
Consider:
$$\xymatrix{
A\times\Pn\times\hig^{(-n)} \ar[r]^{v} \ar[d]^{\phi} & \HS^{' n}\times\hig^{(-n)}\\
A\times\hig^{(-n)}\\
}
$$
We have $R^{i}\phi_{*}(v^{*}(Q)\otimes O(k))=0$ for all $i>0$ and $k\geq 0$.
\end{lemma}

In fact, by Corollary~\ref{summand}, we see that $Q$ is a direct summmand of $\psi_{*}(Q')$, hence the previous lemma is implied by the following (Notice that since the image of $v$ is contained in $\HS^{' n}$, by Proposition~\ref{rewrite using flag} we only need to work with $\Fl^{'n}$):

\begin{lemma}
Consider:
$$\xymatrix{
A\times(\p)^{n}\times\hig^{(-n)} \ar[r]^{w} \ar[d]^{\eta} & \Fl^{' n}\times\hig^{(-n)} \ar[d]^{\psi}\\
A\times\Pn\times\hig^{(-n)} \ar[r]^{v} \ar[d]^{\phi} & \HS^{' n}\times\hig^{(-n)}\\
A\times\hig^{(-n)}\\
}
$$
Then we have $R^{i}\phi_{*}(\eta_{*}(w^{*}(Q'))\otimes O(k))=0$ for $k\geq 0$ and $i>0$.

\end{lemma}

The rest of this subsection is devoted to the proof of this lemma. The strategy is the following: First we are going to construct a locally free resolution of $w^{*}(Q')$, each term of the resolution have an explicit description. Then we use the following standard fact to finish the proof:
\begin{lemma}\label{spectral sequence}
Let $X\xrightarrow{f}Y$ be a proper morphism of schemes, and $M$ a coherent sheaf on $X$. Suppose we have a resolution of $M$ of the form:
$$0\rightarrow C^{-n}\rightarrow\cdots\rightarrow C^{0}\rightarrow 0$$
If $R^{i}f_{*}C^{-j}=0$ for $i>j$, then $R^{p}f_{*}M=0$ for $p>0$.
\end{lemma}

The resolution is constructed in the following way. Recall that we have sheaf $Q'$ on $\Fl'^{n}\times\hig^{(-n)}$ defined by formula~\ref{Q'} in subsection $1.3$. Denote the projection $S\rightarrow \p$ by $q$. For any Higgs bundle $F$ with value in $O(n)$, if we also consider it as a coherent sheaf on the surface $S$, then by Lemma~\ref{resolution} we have a locally free resolution of $F$ on $S$ given by:
$$0\rightarrow q^{*}F(-n)\rightarrow q^{*}F\rightarrow F\rightarrow 0$$
We have the following lemma:
\begin{lemma}\label{resolution Q'}
Let $F$ be the universal Higgs bundle on $\p\times\hig^{(-n)}$, also considered as a coherent sheaf on $S$. Consider:
$$A\times(\p)^{n}\times\hig^{(-n)}\xrightarrow{w}\Fl'^{n}\times\hig^{(-n)}\xrightarrow{\sigma} S^{n}\times\hig^{(-n)}$$
Set $K=w^{*}\sigma^{*}(q^{*}F(-n)\rightarrow q^{*}F)^{\boxtimes n}$. Then $K\otimes O_{\p}(n-1)^{\boxtimes n}$ is a locally free resolution of $w^{*}(Q')$.
\end{lemma}
\begin{proof}
Consider:
$$\xymatrix{
A\times(\p)^{n}\times\hig^{(-n)} \ar[r]^{w} \ar[d]^{\eta} & \Fl^{' n}\times\hig^{(-n)} \ar[d]^{\psi}\\
A\times\Pn\times\hig^{(-n)} \ar[r]^{v} \ar[d]^{\phi} & \HS^{' n}\times\hig^{(-n)}\\
A\times\hig^{(-n)}\\
}
$$
Since $q^{*}F(-n)\rightarrow q^{*}F$ is a resolution of $F$ on $S\times\hig^{(-n)}$, hence
$$(q^{*}F(-n)\rightarrow q^{*}F)^{\boxtimes n}$$
is a resolution of $F^{\boxtimes n}$ on $S^{n}\times\hig^{(-n)}$. It is proved in $[1]$(Proposition $4.2$) that
$$L\sigma^{*}F^{\boxtimes n}\simeq \sigma^{*}F^{\boxtimes n}$$
Hence
$$\sigma^{*}((q^{*}F(-n)\rightarrow q^{*}F)^{\boxtimes n})$$
is a resolution of $\sigma^{*}(F^{\boxtimes n})$ on $\Fl^{' n}\times\hig^{(-n)}$. By Lemma~\ref{factorization to Pn}, $v$ is smooth. Since the square is Cartesian, $w$ is also smooth, hence
$$K=w^{*}\sigma^{*}(q^{*}F(-n)\rightarrow q^{*}F)^{\boxtimes n}$$
is a resolution of $w^{*}\sigma^{*}(F^{\boxtimes n})$. By construction,
$$Q'\simeq \sigma^{*}(F^{\boxtimes n})\otimes\psi^{*}(p_{1}^{*}\de(\mathcal{A})^{-1})$$
Note that
$$w^{*}\psi^{*}(p_{1}^{*}\de(\mathcal{A})^{-1})\simeq \eta^{*}v^{*}(\de(\mathcal{A})^{-1})$$
Also we have
$$\eta^{*}v^{*}(\de(\mathcal{A})^{-1})\simeq \eta^{*}(\de(O_{D})^{-1})$$
by part $(3)$ Corollary~\ref{Hilbert scheme of P1 and S}, where $D$ is the universal subscheme of $\p$ over $\Pn$. Also in Lemma~\ref{det} we already showed that $\de(O_{D})^{-1}\simeq O(n-1)$, hence the result follows from this.

\end{proof}

Let us denote the line bundle $O(a_{1})\boxtimes\cdots\boxtimes O(a_{n})$ on $(\p)^{n}$ by $O(a_{1},\cdots,a_{n})$. We have the following explicit description about the terms in $K$:
\begin{lemma}\label{K}
The complex $K$ sits in nonpositive degree and each term $K^{-p}$ is the direct sum of line bundles of the form $O(a_{1},\cdots,a_{n})$ such that there are $p$ of the $a_{i}$'s equal to $-2n$ and $n-p$ of the $a_{i}$'s equal to $-n$.
\end{lemma}

\begin{proof}
This follows from the fact that the composition:
$$A\times(\p)^{n}\times\hig^{(-n)}\xrightarrow{w} \Fl'^{n}\times\hig^{(-n)}\xrightarrow{\sigma} S^{n}\times\hig^{(-n)}\xrightarrow{q} (\p)^{n}\times\hig^{(-n)}$$
agrees with the projection:
$$A\times(\p)^{n}\times\hig^{(-n)}\rightarrow (\p)^{n}\times\hig^{(-n)}$$
Since $K=w^{*}\sigma^{*}(F(-n)\rightarrow F)^{\boxtimes n}$, so each term in the complex $K$ is the pullback of vector bundles from the projection
$$A\times(\p)^{n}\times\hig^{(-n)}\rightarrow (\p)^{n}\times\hig^{(-n)}$$
Also $F\simeq O(-n)\oplus O(-n)$ by the definition of $\hig^{(-n)}$. Hence the assertion follows from this.
\end{proof}

Now we can prove Lemma~\ref{cohomological vanishing} from the description of $K$:
\begin{proof}
In Lemma~\ref{resolution Q'} we already proved that $K\otimes O_{\p}(n-1)^{\boxtimes n}$ is a resolution of $w^{*}(Q')$. So from spectral sequence argument, we only need to prove
$$R^{i}(\phi\eta)_{*}(K^{-p}\otimes O_{\p}(k)^{\boxtimes n}\otimes O_{\p}(n-1)^{\boxtimes n})=0$$
for all $i>p$ and $k\geq 0$ by Lemma~\ref{spectral sequence}. This follows from the previous lemma.

\end{proof}

\section{Proof of the main theorem}

In the last section we are going to prove Part $(2)$ of Theorem~\ref{main theorem}. First let us reformulate it in the following way: Recall that in Corollary~\ref{part 1 of the main theorem} we already showed that $g^{*}(Q)$ is a maximal Cohen-Macaulay sheaf on $\blA\times_{H}\hig^{(-n)}$. Denote it by $M$. Then we can restate part $(2)$ of the main theorem in the following way:
\begin{theorem}\label{reformulation of main theorem}
Consider:
$$\xymatrix{
\blA\times_{H}\hig^{(-n)} \ar[d]^{\pi}\\
A\times_{H}\hig^{(-n)}\\
}
$$
We have that:
\begin{enumerate1}
\item $R\pi_{*}(M)$ is a sheaf, i.e. $R^{i}\pi_{*}(M)=0$ for $i>0$.
\item $\RH(R\pi_{*}(M),O)$ sits in degree $0$.
\end{enumerate1}
\end{theorem}

Assume the next lemma, we can prove part $(1)$:
\begin{proof}
We have a commutative diagram:
$$\xymatrix{
\blA\times_{H}\hig^{(-n)} \ar[r] \ar[d]^{\pi} & \blA\times\hig^{(-n)} \ar[d]^{\pi'}\\
A\times_{H}\hig^{(-n)} \ar[r] & A\times\hig^{(-n)}\\
}
$$
By part $(2)$ of Lemma~\ref{cohomology of Q} and Lemma~\ref{spectral sequence}, we need to prove that
\begin{equation}\label{vanishing}
R^{i}\pi'_{*}(\bigwedge^{p}(\check{\mathcal{E}})\otimes p^{*}(Q))=0
\end{equation}
for $i>p$. Part $(2)$ of Lemma~\ref{cohomology of Q} implies that $\bigwedge^{p}(\check{\mathcal{E}})$ has a filtration with subquotient $O(-kE)$ with $k\geq 0$, hence Part $(1)$ of Lemma~\ref{cohomology of Q} implies equation~\ref{vanishing}.

\end{proof}

\begin{lemma}\label{cohomology of Q}
Recall that the sheaf $p^{*}(Q)$ on $\blA\times\hig^{(-n)}$, which is a maximal Cohen-Macaulay sheaf supported on $W$ by Corollary~\ref{pullback of Q}. Consider:
$$\xymatrix{
\blA\times\hig^{(-n)} \ar[d]^{\pi'}\\
A\times\hig^{(-n)}\\
}
$$
Then we have:
\begin{enumerate1}
\item $R^{i}\pi'_{*}(p^{*}(Q)\otimes O(-aE))=0$ for all $i>0$ and $a\geq 0$.
\item there exist a resolution of $M$ of the form $\bigwedge^{*}(\check{\mathcal{E}})\otimes p^{*}(Q)$ on $\blA\times\hig^{(-n)}$. Moreover, $\mathcal{E}$ sits in an exact sequence of the form:
    $$0\rightarrow O^{n+1}(E)\rightarrow\mathcal{E}\rightarrow O^{n+1}\rightarrow 0$$
\end{enumerate1}
\end{lemma}

\begin{proof}
For part $(1)$, notice that Lemma~\ref{factorization to Pn} implies that we have the following diagram:
$$\xymatrix{
\blA\times\hig^{(-n)} \ar[r]^{u} \ar@/_1pc/[rr]_{p}& A\times\Pn\times\hig^{(-n)} \ar[r]^{v} & \HS^{' n}\times\hig^{(-n)}\\
}
$$
Since $Lp^{*}(Q)\simeq p^{*}(Q)$, and $v$ is smooth, we see that $Lu^{*}(v^{*}(Q))\simeq p^{*}(Q)$. Hence $$Lu^{*}(v^{*}(Q)\otimes O(a))\simeq p^{*}(Q)\otimes O(-aE)$$
Now part $(1)$ follows from Lemma~\ref{cohomological vanishing} and Lemma~\ref{pushforward of bl}.\\
For Part $(2)$, notice that in Lemma~\ref{Koszul resolution on W} we already showed the existence of $\mathcal{E}$ with desired properties. Hence on $W$ we have a Koszul resolution of $O_{\blA\times_{H}\hig^{(-n)}}$ of the form:
$$\bigwedge^{*}(\check{\mathcal{E}})|_{W}\rightarrow O_{\blA\times_{H}\hig^{(-n)}}$$
Since $p^{*}(Q)$ is a maximal Cohen-Macaulay sheaf on $W$, and $\blA\times_{H}\hig^{(-n)}$ is a regular embedding in $W$, hence Corollary~\ref{resolution of CM sheaf} implies the result.

\end{proof}

Next we are going to prove part $(2)$. As we shall see, part $(2)$ follows essentially from part $(1)$ because of the next lemma. In order to state it, we first define some morphisms between the stack of Higgs bundles. Let $\alpha$ be the involution of $\hig$ given by:
$$\hig\xrightarrow{\alpha}\hig\qquad F\rightarrow \check{F}=\RH(F,O_{C})=\mathcal{H}om_{O_{C}}(F,O_{C})$$
Observe that by Lemma~\ref{description of the dual} implies that:
$$F\in\hig^{(-n)}\Rightarrow\check{F}\in\hig^{(0)}$$
So $\alpha$ defines an isomorphism: $$\hig^{(-n)}\xrightarrow{\alpha}\hig^{(0)}:\qquad F\rightarrow \mathcal{H}om_{O_{C}}(F,O_{C})=\check{F}$$
We also have an isomorphism
$$\hig^{(0)}\xrightarrow{\beta}\hig^{(-n)}:\qquad F\rightarrow F\otimes O(-n)$$
Hence we have induced isomorphisms:
$$\blA\times_{H}\hig^{(-n)}\xrightarrow{\alpha'}\blA\times_{H}\hig^{(0)}\xrightarrow{\beta'}\blA\times_{H}\hig^{(-n)}$$
$$A\times_{H}\hig^{(-n)}\xrightarrow{\alpha''}A\times_{H}\hig^{(0)}\xrightarrow{\beta''}A\times_{H}\hig^{(-n)}$$

Now we can state:
\begin{lemma}\label{dual of M}
$\RH_{\blA\times_{H}\hig^{(-n)}}(M,O(nE))\simeq (\beta'\alpha')^{*}M$
\end{lemma}
Assume this for the moment, we are going to prove part $(2)$

\begin{proof}(Of part $(2)$ of Theorem~\ref{reformulation of main theorem})
By Grothendieck duality, we have
$$\RH(R\pi_{*}M,O)\simeq R\pi_{*}\RH(M,\omega)$$
where $\omega$ is the relative dualizing sheaf of $\pi$. From Proposition~\ref{dualizing sheaf of bl}, $\omega\simeq O(nE)$, hence
$$\RH(R\pi_{*}M,O)\simeq R\pi_{*}\RH(M,O(nE))$$
By the previous lemma, we have
$$\RH(M,O(nE))\simeq (\beta'\alpha')^{*}M$$
Because the following diagram is commutative:
$$\xymatrix{
\blA\times_{H}\hig^{(-n)} \ar[d]^{\pi} \ar[r]^{\beta'\alpha'} & \blA\times_{H}\hig^{(-n)} \ar[d]^{\pi}\\
A\times_{H}\hig^{(-n)} \ar[r]^{\beta''\alpha''} & A\times_{H}\hig^{(-n)}\\
}
$$
We have:
$$\RH(R\pi_{*}M,O)\simeq R\pi_{*}(\beta'\alpha')^{*}M\simeq (\beta''\alpha'')^{*}R\pi_{*}M$$
Part $(1)$ of Theorem~\ref{reformulation of main theorem} implies that $R\pi_{*}M$ sits in degree $0$, so we are done.

\end{proof}

To prove Lemma~\ref{dual of M}, we need the following result in $[1]$, which describes the following symmetry satisfied by the Poincar\'e sheaf:
\begin{Proposition}(Lemma $6.2$ of $[1]$)\label{dual}
Consider the involution:
$$\overline{J}\xrightarrow{\alpha}\overline{J}$$
defined by
$$F\rightarrow \check{F}=\RH(F,O_{C})$$
Then we have: $$(id\times\alpha)^{*}\mathcal{P}\simeq\check{\mathcal{P}}=\RH(\mathcal{P},O)=\mathcal{H}om(\mathcal{P},O)$$
\end{Proposition}
We also need the following lemma:
\begin{lemma}\label{preparation for the dual}
Recall the line bundle $\mathcal{P}_{L}$ we discussed in Lemma~\ref{equivariance property}. Set $L=O(-n)$. Consider the following diagram, $h$ is given by $D\rightarrow \check{I}_{D}$ where $$\check{I}_{D}=\RH_{O_{C}}(I_{D},O_{C})=\mathcal{H}om_{O_{C}}(I_{D},O_{C})$$
$$
\xymatrix{
\blA \ar[r]^{p} \ar@/_1pc/[rr]_{q} & Hilb_{\mathcal{C}|H}^{n} \ar[r]^{h} &\widetilde{\hig}\\
}
$$
Then we have:
\begin{enumerate1}
\item $q^{*}(\mathcal{P}_{L})\simeq O(nE)$.
\item The image of $p$ lies in $U_{n}$(Recall the definition of $U_{n}$ in Proposition~\ref{definition of Q}).
\item The maximal Cohen-Macaulay sheaf $M$ on $\blA\times_{H}\hig^{(-n)}$ is the pullback of the Poincar\'e sheaf on $\widetilde{\hig}\times_{H}\hig^{(-n)}$(See Corollary~\ref{CM sheaf constructed in 1}) via:
    $$\blA\times_{H}\hig^{(-n)}\xrightarrow{q\times id}\widetilde{\hig}\times_{H}\hig^{(-n)}$$
\item The morphism $q$ is a local complete intersection morphism.
\end{enumerate1}
\end{lemma}

\begin{proof}
Part $(1)$ is a direct consequence of Proposition~\ref{line bundle}. In fact, on $\p\times\blA$ we have:
$$0\rightarrow K\rightarrow O_{C}\rightarrow O_{D}\rightarrow 0$$
$K$ sits in an exact sequence(Lemma~\ref{description of kernel}):
$$0\rightarrow O(E)\otimes O(-n)\rightarrow K\rightarrow O(-n)\rightarrow 0$$
Moreover, by Lemma~\ref{description of the dual} we have:
$$\mathcal{H}om_{O_{C}}(K,O_{C})\simeq \mathcal{H}om_{O_{\p}}(K,O_{\p})\otimes O(-n)$$
So if we write $O(-n)\simeq O(-nx_{0})$ for some $x_{0}\in\p$, then by Proposition~\ref{line bundle}, $q^{*}(\mathcal{P}_{L})\simeq \de(\check{K}_{nx_{0}})^{-1}\simeq O(nE)$.\\
For part $2$, notice that in our case $I_{D}=K$, and we have an exact sequence (Lemma~\ref{description of kernel}):
$$0\rightarrow O(E)\otimes O(-n)\rightarrow K\rightarrow O(-n)\rightarrow 0$$
Hence by Lemma~\ref{description of the dual}, $\check{I}_{D}$ sits in an exact sequence:
$$0\rightarrow O\rightarrow \check{I}_{D}\rightarrow O(E)\rightarrow 0$$
So if we consider:
$$\xymatrix{
\p\times\blA \ar[d]^{f}\\
\blA\\
}
$$
Then for any $y\in\blA$, we have the following exact sequence on the fiber over $y$:
$$0\rightarrow O\rightarrow (\check{I}_{D})_{y}\rightarrow O\rightarrow 0$$
Hence over each fiber, $(\check{I}_{D})_{y}\simeq O\oplus O$, so the claim follows from this.\\
For part $(3)$, notice that by the construction of $M$ and part $(2)$, $M$ is the pullback of the maximal Cohen-Macaulay sheaf $Q$ on $$U_{n}\times_{H}\hig\subseteq\textrm{Hilb}^{n}_{\mathcal{C}|H}\times_{H}\hig^{(-n)}$$
via:
$$\blA\times_{H}\hig^{(-n)}\xrightarrow{p}U_{n}\times_{H}\hig^{(-n)}$$
By Proposition~\ref{definition of Q}, $Q$ is the pullback of the Poincar\'e sheaf $\mathcal{P}$ on $\widetilde{\hig}\times_{H}\hig^{(-n)}$, this implies $(3)$\\
For part $(4)$, notice that the image of $q$ is contained in $\hig^{(0)}\bigcap\widetilde{\hig}$. In fact, in the proof of part $(2)$ we have already seen that for any $y\in\blA$, we have the following exact sequence on the fiber over $y$:
$$0\rightarrow O\rightarrow (\check{I}_{D})_{y}\rightarrow O\rightarrow 0$$
Hence it follows from the definition of $\hig^{(0)}$ that the image of $q$ is in $\hig^{(0)}$. Now since $\hig^{(0)}$ is smooth (Proposition~\ref{smoothness of semistable higgs bundles}), and $\blA$ is also smooth, hence $q$ is a local complete intersection morphism by Part $(4)$ of Proposition~\ref{properties of lci}.

\end{proof}

Now we can prove Lemma~\ref{dual of M}:
\begin{proof}
First by part $(3)$ of Lemma~\ref{preparation for the dual}, the sheaf $M$ on $\blA\times_{H}\hig^{(-n)}$ is the pullback of the Poincar\'e sheaf $\mathcal{P}$ on $\widetilde{\hig}\times_{H}\hig^{(-n)}$ via:
$$\blA\times_{H}\hig^{(-n)}\xrightarrow{q\times id} \widetilde{\hig}\times_{H}\hig^{(-n)}$$
We have the following commutative diagram:
\[\begin{CD}
\blA\times_{H}\hig^{(-n)}@>{id_{\blA}\times \beta\alpha}>>\blA\times_{H}\hig^{(-n)}\\
@V{q\times id}VV@V{q\times id}VV\\
\widetilde{\hig}\times_{H}\hig^{(-n)}@>{id_{\widetilde{\hig}}\times \beta\alpha}>>\widetilde{\hig}\times_{H}\hig^{(-n)}\\
\end{CD}\]
Since $\beta$ is given by tensoring with the line bundle $L=O(-n)$, combining Lemma~\ref{equivariance property}, Proposition~\ref{pullback of CM sheaf} and Proposition~\ref{dual}, we have:
$$(id_{\widetilde{\hig}}\times\beta\alpha)^{*}(\mathcal{P})\simeq\check{\mathcal{P}}\otimes p_{1}^{*}\mathcal{P}_{L}$$
where $p_{1}$ is the projection $\widetilde{\hig}\times_{H}\hig^{(-n)}\rightarrow \widetilde{\hig}$. Hence:
$$(id_{\blA}\times\beta\alpha)^{*}M\simeq (q\times id)^{*}(id_{\widetilde{\hig}}\times\beta\alpha)^{*}(\mathcal{P})\simeq (q\times id)^{*}(\check{\mathcal{P}}\otimes p_{1}^{*}\mathcal{P}_{L})$$
By part $(1)$ of Lemma~\ref{preparation for the dual}, we have
$$(q\times id)^{*}(p_{1}^{*}\mathcal{P}_{L})\simeq O(nE)$$
Also by Proposition~\ref{pullback of CM sheaf} and part $(4)$ of Lemma~\ref{preparation for the dual} and Proposition~\ref{pullback of CM sheaf} we have:
$$(q\times id)^{*}(\check{\mathcal{P}})\simeq\RH(M,O)$$
hence the result follows.

\end{proof}

\end{document}